\documentclass[12pt,english,a4paper]{smfart}

\usepackage[T1]{fontenc}
\usepackage{lmodern}
\usepackage{smfthm}
\usepackage[headings]{fullpage}
\usepackage{amssymb,amscd}

\newcommand*{\Qp}{\mathbf{Q}_p}
\newcommand*{\Cp}{\mathbf{C}_p}
\newcommand*{\Zp}{\mathbf{Z}_p}

\newcommand*{\ZZ}{\mathbf{Z}}
\newcommand*{\OO}{\mathcal{O}}
\newcommand*{\MM}{\mathfrak{m}}
\newcommand*{\Qpbar}{\overline{\mathbf{Q}}_p}
\renewcommand{\phi}{\varphi}
\newcommand*{\phiq}{\varphi_q}
\newcommand*{\psiq}{\psi_q}

\renewcommand{\projlim}{\varprojlim}
\renewcommand{\geq}{\geqslant}
\renewcommand{\leq}{\leqslant} 
\newcommand*{\calR}{\mathcal{R}}

\renewcommand{\O}{\mathrm{O}}

\newcommand*{\Gal}{\mathrm{Gal}}
\newcommand*{\Hom}{\mathrm{Hom}}
\newcommand*{\ind}{\mathrm{ind}}
\newcommand*{\cor}{\mathrm{cor}}

\newcommand*{\res}{\mathrm{res}}
\newcommand*{\Mat}{\mathrm{Mat}}
\newcommand*{\End}{\mathrm{End}}

\newcommand*{\GL}{\mathrm{GL}}
\newcommand*{\Id}{\mathrm{Id}}
\newcommand*{\Nm}{\mathrm{N}}

\newcommand*{\Col}{\operatorname{Col}}
\newcommand*{\Exp}{\operatorname{Exp}}

\newcommand*{\Fil}{\mathrm{Fil}}
\newcommand*{\val}{\mathrm{val}}
\newcommand*{\vp}{\mathrm{val}_p}

\newcommand*{\an}{\mathrm{an}}

\newcommand*{\unr}{\mathrm{unr}}
\newcommand*{\cyc}{\mathrm{cyc}}
\newcommand*{\bfont}{\mathbf{B}}
\newcommand*{\afont}{\mathbf{A}}

\newcommand*{\bdr}{\mathbf{B}_{\mathrm{dR}}}  
\newcommand*{\bcris}[1]{\mathbf{B}_{\mathrm{cris}, #1}}  
\newcommand*{\bmax}[1]{\mathbf{B}_{\max, #1}}  
\newcommand*{\btrig}[2]{\widetilde{\mathbf{B}}^{\dagger #1}_{\mathrm{rig} #2}} 
\newcommand*{\brig}[2]{\mathbf{B}^{\dagger #1}_{\mathrm{rig} #2}}
\newcommand*{\brigplus}[1]{\mathbf{B}^{+}_{\mathrm{rig} #1}}

\newcommand*{\btdag}[1]{\widetilde{\mathbf{B}}^{\dagger #1}}
\newcommand*{\bdag}[1]{\mathbf{B}^{\dagger #1}}

\renewcommand{\ddag}[1]{\mathrm{D}^{\dagger #1}}

\newcommand*{\drig}[1]{\mathrm{D}^{\dagger #1}_{\mathrm{rig}}}

\newcommand*{\dfont}{\mathrm{D}}

\newcommand*{\dcris}{\mathrm{D}_{\mathrm{cris}}}

\newcommand{\dpar}[1]{(\!( #1 )\!)}
\newcommand{\dcroc}[1]{[\![ #1 ]\!]}
\newcommand*{\LT}{\operatorname{LT}}
\newcommand*{\Gm}{\mathbf{G}_{\mathrm{m}}}
\newcommand*{\Log}{\operatorname{Log}}
\newcommand*{\chicyc}{\chi_{\cyc}}
\newcommand*{\chilt}{\chi_\pi}

\newcommand*{\Tr}{\mathrm{Tr}}

\newcommand*{\TT}{\mathrm{Tr}^{\LT}}

\newcommand*{\cH}{\mathrm{H}}
\newcommand*{\cB}{\mathrm{B}}
\newcommand*{\cZ}{\mathrm{Z}}

\newcommand{\pmat}[1]{\begin{pmatrix} #1 \end{pmatrix}}
\newcommand*{\Han}{\mathrm{H}_{\mathrm{an}}}
\newcommand*{\Hiw}{\mathrm{H}_{\mathrm{Iw}}}
\newcommand*{\Zan}{\mathrm{Z}_{\mathrm{an}}}
\newcommand*{\Ban}{\mathrm{B}_{\mathrm{an}}}
\newcommand*{\coker}{\operatorname{coker}}
\newcommand*{\Res}{\operatorname{Res}}

\author{Laurent Berger}
\address{UMPA de l'ENS de Lyon \\
UMR 5669 du CNRS \\ Universit\'e de Lyon}
\email{laurent.berger@ens-lyon.fr}
\urladdr{perso.ens-lyon.fr/laurent.berger/}

\author{Lionel Fourquaux}
\address{IRMAR \\ UMR 6625 du CNRS \\ Universit\'e Rennes 1}
\email{lionel.fourquaux+ovrlt@normalesup.org}
\urladdr{www.normalesup.org/~fourquau/pro/}
 
\date{\today}

\title[Iwasawa theory and $F$-analytic Lubin-Tate $(\varphi,\Gamma)$-modules]
{Iwasawa theory and $F$-analytic Lubin-Tate $(\varphi,\Gamma)$-modules}



\NumberTheoremsIn{subsection}

\begin{document}

\begin{abstract}
Let $K$ be a finite extension of $\Qp$. We use the theory of $(\varphi,\Gamma)$-modules in the Lubin-Tate setting to construct some corestriction-compatible families of classes in the cohomology of $V$, for certain representations $V$ of $\Gal(\Qpbar/K)$. If in addition $V$ is crystalline, we describe these classes explicitly using Bloch-Kato's exponential maps. This allows us to generalize Perrin-Riou's period map to the Lubin-Tate setting.
\end{abstract}

\maketitle

\setcounter{tocdepth}{2}

\tableofcontents

\setlength{\baselineskip}{18pt}

\section*{Introduction}

Let $K$ be a finite extension of $\Qp$ and let $G_K = \Gal(\Qpbar/K)$. In this article, we use the theory of $(\varphi,\Gamma)$-modules in the Lubin-Tate setting to construct some classes in $\cH^1(K,V)$, for ``$F$-analytic'' representations $V$ of $G_K$. If in addition $V$ is crystalline, we describe these classes explicitly using Bloch and Kato's exponential maps and generalize Perrin-Riou's period map to the Lubin-Tate setting.

We now describe our constructions in more detail, and introduce some notation which is used throughout this paper. Let $F$ be a finite Galois extension of $\Qp$, with ring of integers $\OO_F$ and maximal ideal $\MM_F$, let $\pi$ be a uniformizer of $\OO_F$ and let $k_F=\OO_F/\pi$ and $q = \mathrm{Card}(k_F)$. Let $\LT$ be the Lubin-Tate formal group \cite{LT65} attached to $\pi$. We fix a coordinate $T$ on $\LT$, so that for each $a \in \OO_F$ the multiplication-by-$a$ map is given by a power series $[a](T)=aT+\mathrm{O}(T^2) \in \OO_F\dcroc{T}$. Let $\log_{\LT}(T)$ denote the attached logarithm and $\exp_{\LT}(T)$ its inverse for the composition. Let $\chilt : G_F \to \OO_F^\times$ be the attached Lubin-Tate character. If $K$ is a finite extension of $F$, let $K_n = K(\LT[\pi^n])$ and $K_\infty = \cup_{n \geq 1} K_n$ and $\Gamma_K = \Gal(K_\infty/K)$.

Let $\afont_F$ denote the set of power series $\sum_{i \in \ZZ} a_i T^i$ with $a_i \in \OO_F$ such that $a_i \to 0$ as $i \to -\infty$ and let $\bfont_F = \afont_F[1/\pi]$, which is a field. It is endowed with a Frobenius map $\phiq : f(T) \mapsto f([\pi](T))$ and an action of $\Gamma_F$ given by $g : f(T) \mapsto f([\chilt(g)](T))$. If $K$ is a finite extension of $F$, the theory of the field of norms (\cite{FW2,FW1} and \cite{WCN}) provides us with a finite unramified extension $\bfont_K$ of $\bfont_F$. Recall \cite{F90} that a $(\phi,\Gamma)$-module over $\bfont_K$ is a finite dimensional $\bfont_K$-vector space endowed with a compatible Frobenius map $\phiq$ and action of $\Gamma_K$. We say that a $(\phi,\Gamma)$-module over $\bfont_K$ is \'etale if it has a basis in which $\Mat(\phiq) \in \GL_d(\afont_K)$. The relevance of these objects is explained by the result below (see \cite{F90}, \cite{KR09}).

\begin{enonce*}{Theorem}
There is an equivalence of categories between the category of $F$-linear representations of $G_K$ and the category of \'etale $(\phi,\Gamma)$-modules over $\bfont_K$.
\end{enonce*}

Let $\bfont_F^\dagger$ denote the set of power series $f(T) \in \bfont_F$ that have a non-empty domain of convergence. The theory of the field of norms again provides us \cite{SM95} with a finite extension $\bfont_K^\dagger$ of $\bfont_F^\dagger$. We say that a $(\phi,\Gamma)$-module over $\bfont_K$ is overconvergent if it has a basis in which $\Mat(\phiq) \in \GL_d(\bfont_K^\dagger)$ and $\Mat(g) \in \GL_d(\bfont_K^\dagger)$ for all $g \in \Gamma_K$. If $F=\Qp$, every \'etale $(\varphi,\Gamma)$-module over $\bfont_K$ is overconvergent \cite{CC98}. If $F \neq \Qp$, this is no longer the case \cite{FX13}. Let us say that an $F$-linear representation $V$ of $G_K$ is $F$-analytic if for all embeddings $\tau : F \to \Qpbar$, with $\tau \neq \Id$, the representation $\Cp \otimes_F^\tau V$ is trivial (as a semilinear $\Cp$-representation of $G_K$). The following result is known \cite{PGMLAV}.

\begin{enonce*}{Theorem}
If $V$ is an $F$-analytic representation of $G_K$, it is overconvergent.
\end{enonce*}

Another source of overconvergent representations of $G_K$ is the set of representations that factor through $\Gamma_K$ (see \S\ref{overpgm}). Our first result is the following (theorem \ref{varcolmsurj}).

\begin{enonce*}{Theorem A}
If $V$ is an overconvergent representation of $G_K$, there exists an $F$-analytic representation $X_{\an}$ of $G_K$, a representation $Y_\Gamma$ of $G_K$ that factors through $\Gamma_K$, and a surjective $G_K$-equivariant map $X_{\an} \otimes_F Y_\Gamma \to V$.
\end{enonce*}

We next focus on $F$-analytic representations. Let $\brig{}{,F}$ denote the Robba ring, which is the ring of power series $f(T) = \sum_{i \in \ZZ} a_i T^i$ with $a_i \in F$ such that there exists $\rho<1$ such that $f(T)$ converges for $\rho < |T| < 1$. We have $\bfont_F^\dagger \subset \brig{}{,F}$. The theory of the field of norms again provides us with a finite extension $\brig{}{,K}$ of $\brig{}{,F}$. If $V$ is an $F$-linear representation of $G_K$, let $\dfont(V)$ denote the $(\phi,\Gamma)$-module over $\bfont_K$ attached to $V$. If $V$ is overconvergent, there is a well defined $(\phi,\Gamma)$-module $\ddag{}(V)$ over $\bfont_K^\dagger$ attached to $V$, such that $\dfont(V) = \bfont_K \otimes_{\bfont_K^\dagger} \ddag{}(V)$. We call $\drig{}(V)$  the $(\phi,\Gamma)$-module over $\brig{}{,K}$ attached to $V$, given by $\drig{}(V) = \brig{}{,K} \otimes_{\bfont_K^\dagger} \ddag{}(V)$.

The ring $\brig{}{,K}$ is a free $\phiq(\brig{}{,K})$-module of degree $q$. This allows us to define \cite{FX13} a map $\psiq : \brig{}{,K} \to \brig{}{,K}$ that is a $\Gamma_K$-equivariant left inverse of $\phiq$, and likewise, if $V$ is an overconvergent representation of $G_K$, a map $\psiq : \drig{}(V) \to \drig{}(V)$ that is a $\Gamma_K$-equivariant left inverse of $\phiq$. 

The main result of this article is the construction, for an $F$-analytic representation $V$ of $G_K$, of a collection of maps
\[ h^1_{K_n,V} : \drig{}(V)^{\psiq=1} \to \cH^1(K_n,V), \]
having a certain number of properties. For example, these maps are compatible with corestriction: $\cor_{K_{n+1}/K_n} \circ h^1_{K_{n+1},V} = h^1_{K_n,V}$ if $n \geq 1$. Another property is that if $F=\Qp$ and $\pi=p$ (the cyclotomic case), these maps co\"{\i}ncide with those constructed in \cite{CC99} (and generalized in \cite{LB3}).

If now $K=F$ and $V$ is a crystalline $F$-analytic representation of $G_F$, we give explicit formulas for $h^1_{F_n,V}$ using Bloch and Kato's exponential maps \cite{BK90}. Let $V$ be as above, let $\dcris(V) = (\bcris{F} \otimes_F V)^{G_F}$ (note that because the $\otimes$ is over $F$, this is the identity component of the usual $\dcris$) and let $t_\pi = \log_{\LT}(T)$. Let $\{u_n\}_{n \geq 0}$ be a compatible sequence of primitive $\pi^n$-torsion points of $\LT$. Let $\brigplus{,F}$ denote the positive part of the Robba ring, namely the ring of power series $f(T) = \sum_{i \geq 0} a_i T^i$ with $a_i \in F$ such that $f(T)$ converges for $0 \leq |T| < 1$. If $n \geq 0$, we have a map $\phiq^{-n} : \brigplus{,F} \to F_n \dcroc{t_\pi}$ given by $f(T) \mapsto f(u_n \oplus \exp_{\LT}(t_\pi/\pi^n))$. Using the results of \cite{KR09}, we prove that there is a natural $(\phi,\Gamma)$-equivariant inclusion $\drig{}(V)^{\psiq=1} \to \brigplus{,F}[1/t_\pi] \otimes_F \dcris(V)$. This provides us, by composition, with maps $\phiq^{-n} : \drig{}(V)^{\psiq=1} \to F_n \dpar{t_\pi} \otimes_F \dcris(V)$ and $\partial_V \circ \phiq^{-n} : \drig{}(V)^{\psiq=1} \to F_n \otimes_F \dcris(V)$ where $\partial_V$ is the ``coefficient of $t_\pi^0$'' map. Recall finally that we have two maps, Bloch and Kato's exponential $\exp_{F_n,V} : F_n \otimes_F \dcris(V) \to \cH^1(F_n,V)$ and its dual $\exp^*_{F_n,V^*(1)} \cH^1(F_n,V) \to F_n \otimes_F \dcris(V)$ (the subscript $V^*(1)$ denotes the dual of $V$ twisted by the cyclotomic character, but is merely a notation here). The first result is as follows (theorem \ref{interdual}).

\begin{enonce*}{Theorem B}
If $V$ is as above and $y \in \drig{}(V)^{\psiq=1}$, then
\[ \exp^*_{F_n, V^*(1)}(h^1_{F_n,V}(y)) = 
\begin{cases}
q^{-n} \partial_V(\phiq^{-n}(y)) & \text{if $n \geq 1$} \\
(1-q^{-1}\phiq^{-1})\partial_V(y) &  \text{if $n = 0$.}
\end{cases} \]
\end{enonce*}

Let $\nabla = t_\pi \cdot d/dt_\pi$, let $\nabla_i = \nabla-i$ if $i \in \ZZ$ and let $h \geq 1$ be such that $\Fil^{-h} \dcris(V) = \dcris(V)$. We prove that if $y \in (\brigplus{,F} \otimes_F \dcris(V))^{\psiq=1}$, then $\nabla_{h-1} \circ  \cdots \circ \nabla_0 (y) \in \drig{}(V)^{\psiq=1}$, and we have the following result (theorem \ref{expbk}).

\begin{enonce*}{Theorem C}
If $V$ is as above and $y \in (\brigplus{,F} \otimes_F \dcris(V))^{\psiq=1}$, then
\begin{multline*} 
h^1_{F_n,V}
(\nabla_{h-1} \circ  \cdots \circ \nabla_0 (y)) = 
 (-1)^{h-1} (h-1)!
\begin{cases}
\exp_{F_n,V}(q^{-n} \partial_V(\phiq^{-n}(y))) & \text{if $n \geq 1$} \\
\exp_{F,V}((1-q^{-1}\phiq^{-1})\partial_V(y)) & \text{if $n=0$.}
\end{cases} 
\end{multline*}
\end{enonce*}

Using theorems B and C, we give in \S\ref{bprfan} a Lubin-Tate analogue of Perrin-Riou's ``big exponential map''  \cite{BP94} using the same method as that of \cite{LB3} which treats the cyclotomic case. It will be interesting to compare this big exponential map with the ``big logarithms'' constructed in \cite{F05} and \cite{F08}.

It is also instructive to specialize theorem C to the case $V=F(\chilt)$, which corresponds to ``Lubin-Tate'' Kummer theory. Recall that if $L$ is a finite extension of $F$, Kummer theory gives us a map $\delta : \LT(\MM_L) \to \cH^1(L,F(\chilt))$. When $L$ varies among the $F_n$, these maps are compatible: the diagram
\[ \begin{CD} \LT(\MM_{F_{n+1}}) @>{\delta}>> \cH^1(F_{n+1},V) \\
@V{\TT_{F_{n+1}/F_n}}VV @VV{\cor_{F_{n+1}/F_n}}V \\
\LT(\MM_{F_n}) @>{\delta}>> \cH^1(F_n,V) \end{CD} \]
commutes. Let $S$ denote the set of sequences $\{x_n\}_{n \geq 1}$ with $x_n \in \MM_{F_n}$ and such that $\Tr^{\LT}_{F_{n+1}/F_n}(x_{n+1}) = [q/\pi](x_n)$ for $n \geq 1$. We prove that $S$ is big, in the sense that (if $F \neq \Qp$) the projection on the $n$-th coordinate map $S \otimes_{\OO_F} F \to F_n$ is onto (this would not be the case if we did not have the factor $q/\pi$ in the definition of $S$). Furthermore, we prove that if $x \in S$, there exists a power series $f(T) \in (\brigplus{,F})^{\psiq=1/\pi}$ such that $f(u_n) = \log_{\LT}(x_n)$ for $n \geq 1$. We have $d/dt_\pi(f(T)) \in (\brigplus{,F})^{\psiq=1}$ and the following holds (theorem \ref{expchif}), where $u$ is the basis of $F(\chilt)$ corresponding to the choice of $\{u_n\}_{n \geq 0}$.

\begin{enonce*}{Theorem D}
We have $h^1_{F_n,F(\chilt)} (d/dt_\pi (f(T)) \cdot u) = (q/\pi)^{-n} \cdot \delta(x_n)$ for all $n \geq 1$.
\end{enonce*}

In the cyclotomic case, there is \cite{RC79} a power series $\mathrm{Col}_x(T)$ such that $\Col_x(u_n)=x_n$ for $n \geq 1$. We then have $f(T) = \log \mathrm{Col}_x(T)$, and theorem D is proved in \cite{CC99}. In the general Lubin-Tate case, we do not know whether there is a ``Coleman power series'' of which $f(T)$ would be the $\log_{\LT}$. This seems like a non-trivial question. 

It would be interesting to compare our results with those of \cite{SV}. The authors of \cite{SV} also construct some classes in $\cH^1(K,V)$, but start from the space $\dfont(V(\chilt \cdot \chicyc^{-1}))^{\psiq=\pi/q}$. In another direction, is it possible to extend our constructions to representations of the form $V \otimes_F Y_\Gamma$ with $V$ $F$-analytic and $Y_\Gamma$ factoring through $\Gamma_K$, and in particular recover the explicit reciprocity law of \cite{TT}?

\section{Lubin-Tate $(\phi,\Gamma)$-modules}
\label{ltpgmsec}

In this chapter, we recall the theory of Lubin-Tate $(\phi,\Gamma)$-modules and classify overconvergent representations.

\subsection{Notation}
\label{nota}

Let $F$ be a finite Galois extension of $\Qp$ with ring of integers $\OO_F$, and residue field $k_F$. Let $\pi$ be a uniformizer of $\OO_F$. Let $d=[F:\Qp]$ and $e$ be the ramification index of $F/\Qp$. Let $q=p^f$ be the cardinality of $k_F$ and let $F_0=W(k_F)[1/p]$ be the maximal unramified extension of $\Qp$ inside $F$. Let $\sigma$ denote the absolute Frobenius map on $F_0$.

Let $\LT$ be the Lubin-Tate formal $\OO_F$-module attached to $\pi$ and 
choose a coordinate $T$ for the formal group law, such that the action 
of~$\pi$ on $\LT$ is given by $[\pi](T)=T^q + \pi T$. 
If $a \in \OO_F$, let $[a](T)$ denote the power series that gives the 
action of $a$ on $\LT$. Let $\log_{\LT}(T)$ denote the attached logarithm and $\exp_{\LT}(T)$ its inverse.
If $K$ is a finite extension of $F$, let 
$K_n = K(\LT[\pi^n])$ and let $K_\infty = \cup_{n \geq 1} K_n$. 
Let $H_K=\Gal(\Qpbar/K_\infty)$ and $\Gamma_K = \Gal(K_\infty/K)$. 
By Lubin-Tate theory (see \cite{LT65}), $\Gamma_K$ is isomorphic to an open 
subgroup of $\OO_F^\times$ via the Lubin-Tate character 
$\chilt : \Gamma_K \to \OO_F^\times$. 

Let $n(K) \geq 1$ be such that if $n \geq n(K)$, then $\chilt : \Gamma_{K_n} \to 1+\pi^n \OO_F$ is an isomorphism, and $\log_p : 1+\pi^n \OO_F \to \pi^n \OO_F$ is also an isomorphism.

Since $\log_{\LT}(T)$ converges on the open unit disk, it can be seen as an element of $\brigplus{,F}$ and we denote it by  $t_\pi$. Recall that $g(t_\pi) = \chilt(g) \cdot t_\pi$ if $g \in G_K$ and that $\phiq(t_\pi) = \pi \cdot t_\pi$. Let $\partial = d/dt_\pi$ so that $\partial f (T) = a(T) \cdot df(T)/dT$, where $a(T) = (d \log_{\LT}(T) / dT)^{-1} \in \OO_F \dcroc{T}^\times$. We have $\partial \circ g = \chilt(g) \cdot g \circ \partial$ if $g \in \Gamma_K$ and $\partial \circ \phiq = \pi \cdot \phiq \circ \partial$.

Recall that $\brig{}{,F}$ denotes the Robba ring, the ring of power series $f(T) = \sum_{i \in \ZZ} a_i T^i$ with $a_i \in F$ such that there exists $\rho<1$ such that $f(T)$ converges for $\rho < |T| < 1$. We have $\bdag{}_F \subset \brig{}{,F}$ and by writing a power series as the sum of its plus part and its minus part, we get $\brig{}{,F} = \brigplus{,F} + \bdag{}_F$. 

Each ring $R \in \{ \brig{}{,F}, \brigplus{,F}, \bdag{}_F, \bfont_F\}$ is equipped with a Frobenius map $\phiq : f(T) \mapsto f([\pi](T))$ and an action of $\Gamma_F$ given by $g : f(T) \mapsto f([\chilt(g)](T))$. Moreover, the ring $R$ is a free $\phiq(R)$-module of rank $q$, and we define $\psiq : R \to R$ by the formula $\phiq(\psiq(f)) = 1/q \cdot \Tr_{R/\phiq(R)}(f)$. The map $\psiq$ has the following properties (see for instance \S 2A of \cite{FX13} and \S 1.2.3 of \cite{LT}): $\psiq(x \cdot \phiq(y)) = \psiq(x) \cdot y$, the map $\psiq$ commutes with the action of $\Gamma_F$, $\partial \circ \psiq = \pi^{-1} \cdot \psiq \circ \partial$ and if $f(T) \in \brigplus{,F}$ then $\phiq \circ \psiq(f) = 1/q \cdot \sum_{z \in \LT[\pi]} f(T \oplus z)$. If $M$ is a free $R$-module with a semilinear Frobenius map $\phiq$ such that $\Mat(\phiq)$ is invertible, then any $m \in M$ can be written as $m = \sum_i r_i \cdot \phiq(m_i)$ with $r_i \in R$ and $m_i \in M$ and the map $\psiq : m \mapsto \sum_i \psiq(r_i) \cdot m_i$ is then well-defined. This applies in particular to the rings $\brig{}{,K}$, $\brigplus{,K}$, $\bdag{}_K$, $\bfont_K$ and to the $(\varphi,\Gamma)$-modules over them.

\subsection{Construction of Lubin-Tate $(\phi,\Gamma)$-modules}
\label{ltpgm}

A $(\phi,\Gamma)$-module over $\bfont_K$ (or over $\bdag{}_K$ or over $\brig{}{,K}$) is a finite dimensional $\bfont_K$-vector space $\dfont$ (or a finite dimensional $\bdag{}_K$-vector space or a free $\brig{}{,K}$-module of finite rank respectively), along with a semilinear Frobenius map $\phiq$ whose matrix (in some basis) is invertible, and a continuous, semilinear action of $\Gamma_K$ that commutes with $\phiq$. 

We say that a $(\phi,\Gamma)$-module $\dfont$ over $\bfont_K$ is \'etale if $\dfont$ has a basis in which $\Mat(\phiq) \in \GL_d(\afont_K)$. Let $\bfont$ be the $p$-adic completion of $\cup_{M/F} \bfont_M$ where $M$ runs through the finite extensions of $F$. By specializing the constructions of \cite{F90}, Kisin and Ren prove the following theorem (theorem 1.6 of \cite{KR09}).

\begin{theo}\label{fontkisren}
The functors $V \mapsto \dfont(V) = (\bfont \otimes_F V)^{H_K}$ and $\dfont \mapsto (\bfont \otimes_{\bfont_K} \dfont)^{\phiq=1}$ give rise to mutually inverse equivalences of categories between the category of $F$-linear representations of $G_K$ and the category of \'etale $(\phi,\Gamma)$-modules over $\bfont_K$.
\end{theo}

We say that a $(\phi,\Gamma)$-module $\dfont$ is overconvergent if there exists a basis of $\dfont$ in which the matrices of $\phiq$ and of all $g \in \Gamma_K$ have entries in $\bdag{}_K$.  This basis then generates a $\bdag{}_K$-vector space $\ddag{}$ which is canonically attached to $\dfont$. If $V$ is a $p$-adic representation, we say that it is overconvergent if $\dfont(V)$ is overconvergent, and then $\ddag{}(V)$ denotes the corresponding $(\phi,\Gamma)$-module over $\bdag{}_K$. The main result of \cite{CC98} states that if $F=\Qp$, then every \'etale $(\phi,\Gamma)$-module over $\bfont_K$ is overconvergent (the proof is given for $\pi=p$, but it is easy to see that it works for any uniformizer). If $F \neq \Qp$,  some simple examples (see \cite{FX13}) show that this is no longer the case.

Recall that an $F$-linear representation of $G_K$ is $F$-analytic if $\Cp \otimes_F^\tau V$ is the trivial $\Cp$-semilinear representation of $G_K$ for all embeddings $\tau \neq \Id \in \Gal(F/\Qp)$. This definition is the natural generalization of Kisin and Ren's notion of $F$-crystalline representation. Kisin and Ren then show that if $K \subset F_\infty$, and if $V$ is a crystalline $F$-analytic representation of $G_K$,  the $(\phi,\Gamma)$-module attached to $V$ is overconvergent (see \S 3.3 of \cite{KR09}; they actually prove a stronger result, namely that the $(\phi,\Gamma)$-module attached to such a $V$ is of finite height).

If $\drig{}$ is a $(\phi,\Gamma)$-module over $\brig{}{,K}$, and if $g \in \Gamma_K$ is close enough to $1$, then by standard arguments (see \S 2.1 of \cite{KR09} or \S 1C of \cite{FX13}), the series $\log(g) = \log(1+(g-1))$ gives rise to a differential operator $\nabla_g : \drig{} \to \drig{}$. The map $v \mapsto \exp(v)$ is defined on a neighborhood of $0$ in $\mathrm{Lie}\,\Gamma_K$; the map $\mathrm{Lie}\,\Gamma_K \to \End(\drig{})$ arising from $v \mapsto \nabla_{\exp(v)}$ is $\Qp$-linear, and we say that $\drig{}$ is $F$-analytic if this map is $F$-linear (see \S 2.1 of \cite{KR09} and \S 1.3 of \cite{FX13}).

If $V$ is an overconvergent representation of $G_K$, we let $\drig{}(V) = \brig{}{,K} \otimes_{\bdag{}_K} \ddag{}(V)$. The following is theorem D of \cite{PGMLAV}.

\begin{theo}
\label{eqcat}
The functor $V \mapsto \drig{}(V)$ gives rise to an equivalence of categories between the category of $F$-analytic representations of $G_K$ and the category of \'etale $F$-analytic Lubin-Tate $(\phi,\Gamma)$-modules over $\brig{}{,K}$.
\end{theo}

In general, representations of $G_K$ that are not $F$-analytic are not overconvergent (see \S \ref{overpgm}), and the analogue of theorem \ref{eqcat} without the $F$-analyticity condition on both sides does not hold.

\subsection{Overconvergent Lubin-Tate $(\phi,\Gamma)$-modules}
\label{overpgm}

By theorem \ref{eqcat}, there is an equivalence of categories between the category of $F$-analytic representations of $G_K$ and the category of \'etale $F$-analytic Lubin-Tate $(\phi,\Gamma)$-modules over $\brig{}{,K}$. The purpose of this section is to prove a conjecture of Colmez that describes \emph{all} overconvergent representations of $G_K$. 

Any representation $V$ of $G_K$ that factors through $\Gamma_K$ is overconvergent, since $H_K$ acts trivially on $V$ so that $\dfont(V) = \bfont_K \otimes_F V$ and therefore $\dfont(V)$ has a basis in which $\Mat(\phiq)=\Id$ and $\Mat(g) \in \GL_d(\OO_F)$ if $g \in \Gamma_K$. If $X$ is $F$-analytic and $Y$ factors through $\Gamma_K$, $X \otimes_F Y$ is therefore overconvergent. We prove that any overconvergent representation of $G_K$ is a quotient (and therefore also a subobject, by dualizing) of some representation of the form $X \otimes_F Y$ as above. 

\begin{theo}
\label{varcolmsurj}
If $V$ is an overconvergent representation of $G_K$, there exists an $F$-analytic representation $X$ of $G_K$, a representation $Y$ of $G_K$ that factors through $\Gamma_K$, and a surjective $G_K$-equivariant map $X \otimes_F Y \to V$.
\end{theo}

\begin{proof}
Recall (see \S 3 of \cite{PGMLAV}) that if $r>0$, then inside $\brig{}{,K}$ we have the subring $\brig{,r}{,K}$ of elements defined on a fixed annulus whose inner radius depends on $r$ and whose outer raidus is $1$, and that $(\phi,\Gamma)$-modules over $\brig{}{,K}$ can be defined over $\brig{,r}{,K}$ if $r$ is large enough, giving us a module $\drig{,r}(V)$. We also have rings $\bfont^{[r;s]}_K$ of elements defined on a closed annulus whose radii depend on $r \leq s$. One can think of an element of $\brig{,r}{,K}$ as a compatible family of elements of $\{\bfont^I_K\}_I$ where $I$ runs over a set of closed intervals whose union is $[r;+\infty[$. In the rest of the proof, we use this principle of glueing objects defined on closed annuli to get an object on the annulus corresponding to $\brig{,r}{,K}$.

Choose $r>0$ large enough such that $\drig{,r}(V)$ is defined, and $s \geq qr$. Let $\dfont^{[r;s]}(V) = \bfont^{[r;s]}_K \otimes_{\brig{,r}{,K}} \drig{,r}(V)$. If $a \in \OO_F$, and if $\vp(a) \geq n$ for $n=n(r,s)$ large enough, the series $\exp(a \cdot \nabla)$ converges in the operator norm to an operator on the Banach space $\dfont^{[r;s]}(V)$. This way, we can define a twisted action of $\Gamma_{K_n}$ on $\dfont^{[r;s]}(V)$, by the formula $h \star x = \exp(\log_p(\chilt(h)) \cdot \nabla)(x)$. This action is now $F$-analytic by construction.

Since $s \geq qr$, the modules $\dfont^{[q^m r;q^m s]}(V)$ for $m \geq 0$ are glued together (using the idea explained above) by $\phiq$ and we get a new action of $\Gamma_{K_n}$ on $\drig{,r}(V) = \dfont^{[r;+\infty[}(V)$ and hence on $\drig{}(V)$. Since $\phiq$ is unchanged, this new $(\phi, \Gamma)$-module is \'etale, and therefore corresponds to a representation $W$ of $G_{K_n}$. The representation $W$ is $F$-analytic by theorem \ref{eqcat}, and its restriction to $H_K$ is isomorphic to $V$. 

Let $X = \ind_{G_{K_n}}^{G_K} W$. By Mackey's formula, $X {\mid_{H_K}}$ contains $W{\mid_{H_K}} \simeq V {\mid_{H_K}}$ as a direct summand. The space $Y = \Hom(\ind_{G_{K_n}}^{G_K} W, V)^{H_K}$ is therefore a nonzero representation of $\Gamma_K$, and there is an element $y \in Y$ whose image is $V$. The natural map $X \otimes_F Y \to V$ is therefore surjective. Finally, $X$ is $F$-analytic since $W$ is $F$-analytic.
\end{proof}

By dualizing, we get the following variant of theorem \ref{varcolmsurj}.

\begin{coro}
\label{varcolminj}
If $V$ is an overconvergent representation of $G_K$, there exists an $F$-analytic representation $X$ of $G_K$, a representation $Y$ of $G_K$ that factors through $\Gamma_K$, and an injective $G_K$-equivariant map $V \to X \otimes_F Y$.
\end{coro}

\subsection{Extensions of $(\phi,\Gamma)$-modules}
\label{extpgm}

In this section, we prove that there are no non-trivial extensions between an $F$-analytic $(\phi,\Gamma)$-module and the twist of an $F$-analytic $(\phi,\Gamma)$-module by a character that is not  $F$-analytic. This is not used in the rest of the paper, but is of independent interest.

If $\delta\colon \Gamma_K \to \OO_F^\times$ is a continuous character, and 
$g \in \Gamma_K$, let $w_\delta(g) = \log \delta(g) / \log \chilt(g)$. 
Note that $\delta$ is $F$-analytic if and only if $w_\delta(g)$ is 
independent of $g \in \Gamma_K$. 

We define the first cohomology group $\cH^1(\dfont)$ of a $(\phi,\Gamma)$-module $\dfont$ as in \S 4 
of \cite{FX13}. Let $\dfont$ be a $(\phi,\Gamma)$-module over $\brig{}{,K}$. 
Let $G$ denote the semigroup $\phiq^{\ZZ_{\geq 0}} \times \Gamma_K$ and let 
$\cZ^1(\dfont)$ denote the set of continuous functions 
$f\colon G \to \dfont$ such that $(h-1)f(g) = (g-1)f(h)$ for all 
$g,h \in G$. Let $\cB^1(\dfont)$ be the subset of $\cZ^1(\dfont)$ consisting 
of functions of the form $g \mapsto (g-1)y$, $y \in D$ and let 
$\cH^1(\dfont) = \cZ^1(\dfont) / \cB^1(\dfont)$. If $g \in G$ and $f \in\cZ^1$, then 
$[h \mapsto (g-1)f(h)] = [h \mapsto (h-1)f(g)] \in \cB^1$. The natural 
actions of $\Gamma_K$ and $\phiq$ on $\cH^1$ are therefore trivial.

If $\dfont_0$ and $\dfont_1$ are two $(\phi,\Gamma)$-modules, then $\Hom(\dfont_1, \dfont_0)= \Hom_\textrm{$\brig{}{,K}$-mod}(\dfont_1, \dfont_0)$ is a free $\brig{}{,K}$-module of rank $\operatorname{rk}(\dfont_0)\operatorname{rk}(\dfont_1)$ which is easily seen to be
itself a $(\phi,\Gamma)$-module. The space
$\cH^1(\Hom(\dfont_1, \dfont_0))$ 
classifies the extensions of $\dfont_1$ by $\dfont_0$. More precisely, 
if $\dfont$ is such an extension and if $s\colon \dfont_1 \rightarrow \dfont$ 
is a $\brig{}{,K}$-linear map that is a section of the projection 
$\dfont \rightarrow \dfont_1$, then $g \mapsto s - g(s)$ is a cocycle on $G$ 
with values in $\Hom(\dfont_1, \dfont_0)$ 
(the element $g(s) \in \Hom(\dfont_1, \dfont)$ being defined 
by $g(s)(g(x)) = g(s(x))$ for all $g \in G$ and all $x \in \dfont_1$). The 
class of this cocycle in the quotient 
$\cH^1(\Hom(\dfont_1, \dfont_0))$ 
does not depend on the choice of the section~$s$, and every such class 
defines a unique extension of $\dfont_1$ by $\dfont_0$ up to isomorphism.

\begin{theo}
\label{noext}
If $\dfont$ is an $F$-analytic $(\phi,\Gamma)$-module, and if 
$\delta\colon \Gamma_K \to \OO_F^\times$ is \emph{not} locally 
$F$-analytic, then $\cH^1(\dfont(\delta))=\{0\}$.
\end{theo}

\begin{proof}
If $g \in \Gamma_K$ and $x(\delta) \in \dfont(\delta)$ with $x \in \dfont$, we have 
\[
\nabla_g(x(\delta)) = \nabla(x)(\delta)  + w_\delta(g) \cdot x(\delta).
\]
If $g,h \in \Gamma_K$, this implies that 
$\nabla_g(x(\delta))-\nabla_h(x(\delta)) = (w_\delta(g)-w_\delta(h)) \cdot x(\delta)$. 
If $\overline{f} \in \cH^1(\dfont(\delta))$ and $g \in \Gamma_K$, 
then $g(\overline{f}) = \overline{f}$ and therefore 
$\nabla_g(\overline{f}) = 0$. The formula above shows that if 
$k \in \Gamma_K$, then 
$\nabla_g(f(k))-\nabla_h(f(k)) = (w_\delta(g)-w_\delta(h)) \cdot f(k)$, 
so that 
$0 = (\nabla_g-\nabla_h)(\overline{f}) = (w_\delta(g)-w_\delta(h)) \cdot \overline{f}$, 
and therefore $\overline{f} = 0$ if $\delta$ is not locally analytic.
\end{proof}

\section{Analytic cohomology and Iwasawa theory}
\label{ancohiw}

In this chapter, we explain how to construct classes in the cohomology groups of $F$-analytic $(\phi,\Gamma)$-modules. This allows us to define our maps $h^1_{K_n,V}$.

\subsection{Analytic cohomology}
\label{cohan}

Let $G$ be an $F$-analytic semigroup and let $M$ be a Fr\'echet or LF space with a pro-$F$-analytic (\S 2 of \cite{PGMLAV}) action of $G$. Recall that this means that we can write $M = \varinjlim_i \varprojlim_j M_{ij}$ where $M_{ij}$ is a Banach space with a locally analytic action of $G$. A function $f : G \to M$ is said to be pro-$F$-analytic if its image lies in $\varprojlim_j M_{ij}$ for some $i$ and if the corresponding function $f : G \to M_{ij}$ is locally $F$-analytic for all $j$.

The analytic cohomology groups $\Han^i(G,M)$ are defined and studied in \S 4 of \cite{FX13} and \S 5 of \cite{LT}. In particular, we have $\Han^0(G,M) = M^G$ and $\Han^1(G,M) = \Zan^1(G,M)/\Ban^1(G,M)$ where $\Zan^1(G,M)$ is the set of pro-$F$-analytic functions $f : G \to M$ such that $(g-1)f(h)=(h-1)f(g)$ for all $g,h \in G$ and $\Ban^1(G,M)$ is the set of functions of the form $g \mapsto (g-1)m$.

Let $M$ be a Fr\'echet space, and write $M = \projlim_n M_n$ with $M_n$ a Banach space such that the image of $M_{n+j}$ in $M_n$ is dense for all $j \geq 0$. 

\begin{prop}\label{h1frech}
We have $\Han^1(G,M) = \projlim_n \Han^1(G,M_n)$.
\end{prop}

\begin{proof}
By definition, we have an exact sequence 
\[ 0 \to \Ban^1(G,M_n) \to \Zan^1(G,M_n) \to \Han^1(G,M_n) \to 0. \] 
It is clear that $\Ban^1(G,M) = \projlim_n \Ban^1(G,M_n)$ and that $\Zan^1(G,M) = \projlim_n \Zan^1(G,M_n)$, since these spaces are spaces of functions on $G$ satisfying certain compatible conditions.  The Banach spaces $\Ban^1(G,M_n)$ satisfy the Mittag-Leffler condition: $\Ban^1(G,M_n)=M_n/M_n^G$ and the image of $M_{n+j}$ in $M_n$ is dense for all $j \geq 0$. This implies that the sequence
\[ 0 \to \projlim_n \Ban^1(G,M_n) \to \projlim_n \Zan^1(G,M_n) \to \projlim_n \Han^1(G,M_n) \to 0 \]
is exact, and the proposition follows.
\end{proof}

In this paper, we mainly use the semigroups $\Gamma_K$, $\Gamma_K \times \Phi$ where $\Phi = \{ \phiq^n$, $n \in \ZZ_{\geq 0}\}$ and $\Gamma_K \times \Psi$ where $\Psi = \{ \psiq^n$, $n \in \ZZ_{\geq 0}\}$. The semigroups $\Phi$ and $\Psi$ are discrete and the $F$-analytic structure comes from the one on $\Gamma_K$.

\begin{defi}
\label{defcor}
Let $G$ be a compact group and let $H$ be an open subgroup of $G$. We have the \emph{corestriction} map $\cor : \Han^1(H,M) \to \Han^1(G,M)$, which satisfies $\cor \circ \res = [G:H]$. This map has the following equivalent explicit descriptions (see \S 2.5 of \cite{SCG} and \S II.2 of \cite{CC99}). Let $X \subset G$ be a set of representatives of $G/H$ and let $f \in \Zan^1(H,M)$ be a cocycle.
\begin{enumerate}
\item By Shapiro's lemma, $\Han^1(H,M) = \Han^1(G,\ind_H^G M)$ and $\cor$ is the map induced by $i \mapsto \sum_{x \in X} x \cdot i(x^{-1})$;
\item if $M \subset N$ where $N$ is a $G$-module and if there exists $n \in N$ such that $f(h)=(h-1)(n)$, then $\cor(f)(g)=(g-1)(\sum_{x \in X} xn)$;
\item if $g \in G$, let $\tau_g : X \to X$ be the permutation defined by $\tau_g(x)H=gxH$. We have $\cor(f)(g) = \sum_{x \in X} \tau_g(x) \cdot f( \tau_g(x)^{-1} gx)$.
\end{enumerate}
\end{defi}

If $g \in \Gamma_K$, let $\ell(g) = \log_p \chilt(g)$. If $M$ is a Fr\'echet space with a pro-$F$-analytic action of $\Gamma_K$ and if $g \in \Gamma_K$ is such that $\chilt(g) \in 1+2p \OO_F$, then $\lim_{n \to \infty} (g^{p^n}-1)/(p^n \ell(g))$ converges to an operator $\nabla$ on $M$, which is independent of $g$ thanks to the $F$-analyticity assumption. If $c : \Gamma_K \to M$ is an $F$-analytic map, let $c'(1)$ denote its derivative at the identity.

\begin{prop}\label{lazcor}
If $M$ is a Fr\'echet space with a pro-$F$-analytic action of $\Gamma_K$, the map $c \mapsto c'(1)$ induces an isomorphism $\Han^1(\Gamma_K,M) = (M/\nabla M)^{\Gamma_K}$, under which $\cor_{L/K}$ corresponds to $\Tr_{L/K}$.
\end{prop}

\begin{proof}
Assume for the time being that $M$ is a Banach space. We first show that the map induced by $c \mapsto c'(1)$ is well-defined and lands in $(M/\nabla M)^{\Gamma_K}$. The map $c \mapsto c'(1)$ from $\Zan^1(\Gamma_K,M) \to M$ is well-defined, and if $c(g)=(g-1)m$, then $c'(1) = \nabla m$ so that there is a well-defined map $\Han^1(\Gamma_K,M)  \to M/\nabla M$. If $h \in \Gamma_K$ then $(h-1)c'(1) = \lim_{g \to 1} (h-1)c(g)/\ell(g) = \lim_{g \to 1} (g-1)c(h)/\ell(g) = \nabla c(h)$ so that the image of $c \mapsto c'(1)$ lies in $(M/\nabla M)^{\Gamma_K}$.

The formula for the corestriction follows from the explicit descriptions above: if $h \in \Gamma_L$ then $\tau_h(x) = x$ so that  $\cor(c)(h) = \sum_{x \in X} x \cdot c(h)$ and 
\[ \cor(c)'(1) = \lim_{h \to 1} \cor(c)(h) / \ell(h) = \sum_{x \in X} x \cdot c'(1) = \Tr_{L/K} (c'(1)). \]

We now show that the map is injective. If $c'(1) = \nabla m$, then the derivative of $g \mapsto c(g)-(g-1)m$ at $g=1$ is zero and hence $c(g)=(g-1)m$ on some open subgroup $\Gamma_L$ of $\Gamma_K$ and $c = [L:K]^{-1} \cor_{L/K} \circ \res_{K/L} (c) = 0$.

We finally show that the map is surjective. Suppose now that $y \in (M/\nabla M)^{\Gamma_K}$. The formula $g \mapsto (\exp(\ell(g)\nabla)-1)/\nabla \cdot y$ defines an analytic cocycle $c_L$ on some open subgroup $\Gamma_L$ of $\Gamma_K$. The image of $[L:K]^{-1} c_L$ under $\cor_{L/K}$ gives a cocyle $c \in \Han^1(\Gamma_K,M)$ such that $c'(1) = y$.

We now let $M = \projlim_n M_n$ be a Fr\'echet space. The map $\Han^1(\Gamma_K,M) \to (M/ \nabla M)^{\Gamma_K}$ induced by $c \mapsto c'(1)$ is well-defined, and in the other direction we have the map $y \mapsto c_y$:
\[ (M/ \nabla M)^{\Gamma_K} \to \projlim_n (M_n/ \nabla M_n)^{\Gamma_K} \to \projlim_n \Han^1(\Gamma_K,M_n) \to \Han^1(\Gamma_K,M). \]
These two maps are inverses of each other.
\end{proof}

\begin{rema}\label{liecohlaz}
Compare with the following theorem (see \cite{GT}, corollary 21): if $G$ is a compact $p$-adic Lie group and if $M$ is a locally analytic representation of $G$, then $\Han^i(G,M) = \cH^i(\mathrm{Lie}(G),M)^G$.
\end{rema}

\subsection{Cohomology of $F$-analytic $(\phi,\Gamma)$-modules}
\label{pgmcoh}

If $V$ is an $F$-analytic representation, let $\Han^1(K,V) \subset \cH^1(K,V)$ classify the $F$-analytic extensions of $F$ by $V$. Let $\dfont$ denote an $F$-analytic $(\phi,\Gamma)$-module over $\brig{}{,K}$, such as $\drig{}(V)$.

\begin{prop}\label{h1anextv}
If $V$ is $F$-analytic, then $\Han^1(K,V) = \Han^1(\Gamma_K \times \Phi, \drig{}(V))$.
\end{prop}

\begin{proof}
The group $\Han^1(\Gamma_K \times \Phi, \drig{}(V))$ classifies the $F$-analytic extensions of $\brig{}{,K}$ by $\drig{}(V)$, which correspond to $F$-analytic extensions of $F$ by $V$ by theorem \ref{eqcat}.
\end{proof}

\begin{theo}\label{colpsizero}
If $\dfont$ is an $F$-analytic $(\phi,\Gamma)$-module over $\brig{}{,K}$ and $i  = 0,1$, then $\Han^i(\Gamma_K,\dfont^{\psiq=0}) = 0$.
\end{theo}

\begin{proof}
Since $\brig{}{,F} \subset \brig{}{,K}$, the $\brig{}{,K}$-module $\dfont$ is a free $\brig{}{,F}$-module of finite rank. Let $\calR_F$ denote $\brig{}{,F}$ and let $\calR_{\Cp}$ denote $\Cp \widehat{\otimes}_F \brig{}{,F}$ the Robba ring with coefficients in $\Cp$. There is an action of $G_F$ on the coefficients of $\calR_{\Cp}$ and $\calR_{\Cp}^{G_F}= \calR_F$. 

Theorem 5.5 of \cite{LT} says that $\Han^i(\Gamma_K,(\calR_{\Cp} \otimes_{\calR_F} \dfont)^{\psiq=0}) = 0$. For $i=0$, this implies our claim. For $i=1$, it says that if $c : \Gamma_K \to \dfont^{\psiq=0}$ is an $F$-analytic cocycle, there exists $m \in (\calR_{\Cp} \otimes_{\calR_F} \dfont)^{\psiq=0}$ such that $c(g)=(g-1)m$ for all $g \in \Gamma_K$. If $\alpha \in G_F$, then $c(g)=(g-1)\alpha(m)$ as well, so that $\alpha(m)-m \in ((\calR_{\Cp} \otimes_{\calR_F} \dfont)^{\psiq=0})^{\Gamma_K} = 0$. This shows that $m \in ((\calR_{\Cp} \otimes_{\calR_F} \dfont)^{\psiq=0})^{G_F} = \dfont^{\psiq=0}$.
\end{proof}

\begin{coro}\label{cohphipsi}
The groups $\Han^i(\Gamma_K \times \Phi, \dfont)$ and $\Han^i(\Gamma_K \times \Psi, \dfont)$ are isomorphic for $i  = 0,1$.
\end{coro}

\begin{proof}
If $i=0$, then we have an inclusion $\dfont^{\phiq=1,\Gamma_K} \subset \dfont^{\psiq=1,\Gamma_K}$. If $x \in \dfont^{\psiq=1,\Gamma_K}$, then $x-\phiq(x) \in \dfont^{\psiq=0,\Gamma_K} = \{ 0 \}$ by theorem \ref{colpsizero}, so that $x=\phiq(x)$ and the above inclusion is an equality.

Now let $i = 1$. If $f \in \Zan^1(\Gamma_K \times \Phi, \dfont)$, let $Tf \in \Zan^1(\Gamma_K \times \Psi, \dfont)$ be the function defined by $Tf(g) = f(g)$ if $g \in \Gamma_K$ and $Tf(\psiq) = -\psiq (f(\phiq))$.

If $f \in \Zan^1(\Gamma_K \times \Psi, \dfont)$ and $g \in \Gamma_K$, then $(\phiq \psiq -1) f(g) \in \dfont^{\psiq=0}$ and the map $g \mapsto (\phiq \psiq-1) f(g)$ is an element of $\Zan^1(\Gamma_K,\dfont^{\psiq=0})$. By theorem \ref{colpsizero}, applied once for existence and once for unicity, there is a unique $m_f \in \dfont^{\psiq=0}$ such that $(\phiq \psiq-1) f(g) = (g-1)m_f$. Let $Uf \in \Zan^1(\Gamma_K \times \Phi, \dfont)$ be the function defined by $Uf (g) = f(g)$ if $g \in \Gamma_K$ and $Uf(\phiq) = -\phiq (f(\psiq)) + m_f$. 

It is straightforward to check that $U$ and $T$ are inverses of each other (even at the level of the $\Zan^1$) and that they descend to the $\Han^1$.
\end{proof}

\begin{theo}\label{infrespsi}
The map $f \mapsto f(\psiq)$ from $\Zan^1(\Gamma_K \times \Psi, \dfont)$ to $\dfont$ gives rise to an exact sequence:
\[ 0 \to \Han^1(\Gamma_K, \dfont^{\psiq=1}) \to \Han^1(\Gamma_K \times \Psi, \dfont) \to \left( \frac{\dfont}{\psiq-1} \right)^{\Gamma_K} \]
\end{theo}

\begin{proof}
If $f \in \Zan^1(\Gamma_K \times \Psi, \dfont)$ and $g \in \Gamma_K$, then $(g-1) f(\psiq) = (\psiq-1) f(g) \in (\psiq-1) \dfont$ so that the image of $f$ is in $(\dfont/(\psiq-1))^{\Gamma_K}$. The other verifications are similar.
\end{proof}

\subsection{The space $\dfont/(\psiq-1)$}
By theorem \ref{infrespsi} in the previous section, the cokernel of the map $\Han^1(\Gamma_K, \dfont^{\psiq=1}) \to \Han^1(\Gamma_K \times \Psi, \dfont)$ injects into $(\dfont/(\psiq-1))^{\Gamma_K}$. It can be useful to know that this cokernel is not too large. In this section, we bound $\dfont/(\psiq-1)$ when $\dfont = \brig{}{,F}$, with the action of $\phiq$ twisted by $a^{-1}$, for some $a\in F^\times$.

\begin{theo}
\label{cokrfa}
If $a \in F^\times$, then $\psiq-a : \brig{}{,F} \to \brig{}{,F}$ is onto unless $a=q^{-1} \pi^m$ for some $m \in \ZZ_{\geq 1}$, in which case $\brig{}{,F} / (\psiq-a)$ is of dimension $1$.
\end{theo}

In order to prove this theorem, we need some results about the action of $\psiq$ on $\brig{}{,F}$. Recall that the map $\partial = d/dt_\pi$ was defined in \S \ref{nota}.

\begin{lemm}
\label{fxphiRp}
If $a \in F^\times$, then $a \phiq - 1 : \brigplus{,F} \to \brigplus{,F}$  is an isomorphism, unless $a = \pi^{-m}$ for some $m \in \ZZ_{\geq 0}$, in which case
\begin{align*}
\ker(a \phiq - 1 : \brigplus{,F} \to \brigplus{,F}) &= F t_\pi^m \\
\mathrm{im}(a \phiq - 1 : \brigplus{,F} \to \brigplus{,F})& = \{ f(T) \in \brigplus{,F} \mid \partial^m(f)(0) = 0 \}.
\end{align*}
\end{lemm}

\begin{proof}
This is lemma~5.1 of~\cite{FX13}.
\end{proof}

\begin{lemm}
\label{kpsienough}
If $m \in \ZZ_{\geq 0}$, there is an $h(T) \in (\brigplus{,F})^{\psiq = 0}$ such that $\partial^m(h)(0) \neq 0$.
\end{lemm}

\begin{proof}
We have $\psiq(T) = 0$ by (the proof of) proposition 2.2 of \cite{FX13}. If there was some $m_0$ such that $\partial^m(T)(0) = 0$ for all $m \geq m_0$, then $T$ would be a polynomial in $t_\pi$, which it is not. This implies that there is a sequence $\{m_i\}_i$ of integers with $m_i \to +\infty$, such that $\partial^{m_i}(T)(0) \neq 0$, and we can take $h(T) = \partial^{m_i-m}(T)$ for any $m_i \geq m$.
\end{proof}

\begin{coro}
\label{surjplus}
If $a \in F^\times$, then $\psiq-a : \brigplus{,F} \to \brigplus{,F}$  is onto.
\end{coro}

\begin{proof}
If $f(T) \in \brigplus{,F}$ and if we can write $f = (1- a \phiq)g$, then $f=(\psiq-a)(\phiq(g))$. If this is not possible, then by lemma \ref{fxphiRp} there exists $m \geq 0$ such that $a=\pi^{-m}$ and $\partial^m (f)(0) \neq 0$. Let $h$ be the function provided by lemma \ref{kpsienough}. The function $f- (\partial^m (f)(0) / \partial^m (h)(0)) \cdot h$ is in the image of $1-a \phiq$ by lemma \ref{fxphiRp}, and $h=(\psiq-a)(-a^{-1} h)$ since $\psiq(h)=0$. This implies that $f$ is in the image of $\psiq-a$.
\end{proof}

\begin{lemm}
\label{psiminusbiga}
If $a^{-1} \in q \cdot \OO_F$, then $\psiq - a : \brig{}{,F} \to \brig{}{,F}$ is onto.
\end{lemm}

\begin{proof}
We have $\brig{}{,F} =  \brigplus{,F} + \bdag{}_F$  (by writing a power series as the sum of its plus part and of its minus part) and by corollary \ref{surjplus},  $\psiq-a : \brigplus{,F} \to \brigplus{,F}$  is onto. Take $f(T) \in \bdag{}_F$, choose some $r>0$ and let $\bfont_F^{(0,r]}$ be the set of $f(T) \in \bdag{}_F$ that converge and are bounded on the annulus $0 < \vp(x) \leq r$. It follows from proposition 1.4 of \cite{LT} that if $n \gg 0$, then $\psiq^n(f) \in \bfont_F^{(0,r]}$ and by proposition 2.4(d) of \cite{FX13}, the sequence $(q/\pi \cdot \psiq)^n(f)$ is bounded in $\bfont_F^{(0,r]}$. The series $\sum_{n \geq 0} a^{-1-n} \psiq^n (f)$ therefore converges in $\bfont_F^{(0,r]}$, and we can write $f = (\psiq-a)g$ where $g = a^{-1} (1-a^{-1} \psiq)^{-1} f = \sum_{n \geq 0} a^{-1-n} \psiq^n (f)$.
\end{proof}

Let $\Res : \brig{}{,F} \to F$ be defined by $\Res(f) = a_{-1}$ where $f(T) dt_\pi = \sum_{n \in \ZZ} a_n T^n dT$. The following lemma combines propositions 2.12 and 2.13 of \cite{FX13}.

\begin{lemm}
\label{fxseqres}
The sequence $0 \to F \to \brig{}{,F} \xrightarrow{\partial}
\brig{}{,F} \xrightarrow{\Res} F \to 0$ is exact, and $\Res(\psiq(f)) = \pi/q \cdot \Res(f)$.
\end{lemm}

\begin{proof}[Proof of theorem \ref{cokrfa}]
Since $\partial \circ \psiq = \pi^{-1} \psiq \circ \partial$, the map $\partial$ induces a map:
\begin{equation}\tag{Der}  
\frac{\brig{}{,F}}{\psiq-a} \xrightarrow{\quad\partial\quad} \frac{\brig{}{,F}}{\psiq-a \pi}. 
\end{equation}
Take $x \in \brig{}{,F}$ such that $\Res(x)=1$. We have $\Res((\psiq-a\pi)x) = \pi/q-a\pi$. If $a\neq q^{-1}$, this is non-zero and if $f \in \brig{}{,F}$, proposition \ref{fxseqres} allows us to write $f=\partial g + \Res(f)/(\pi/q-a\pi) \cdot (\psiq-a\pi)x$. This implies that (Der) is onto if $a\neq q^{-1}$.

Combined with lemma \ref{psiminusbiga}, this implies that $\brig{}{,F} / (\psiq-a) = 0$ if $a$ is not of the form $q^{-1} \pi^m$ for some $m \in \ZZ_{\geq 1}$. 

When $a=q^{-1}$, we have an exact sequence
\[ \frac{\brig{}{,F}}{\psiq-q^{-1}} \xrightarrow{\quad \partial \quad} \frac{\brig{}{,F}}{\psiq-q^{-1} \pi}
\xrightarrow{\quad\Res\quad} F \to 0, \]
which now implies that $\brig{}{,F} / (\psiq-q^{-1} \pi) = F$, generated by the class of $x$.

We now assume again that $a \neq q^{-1}$ and compute the kernel of (Der). If $f \in \brig{}{,F}$ is such that $\partial f = (\psiq-a\pi)g$, then $\Res \partial f = \Res (\psiq-a\pi)g  = (\pi/q-a\pi) \Res(g)$, so that $\Res(g)=0$ and we can write $g = \partial h$. We have $\partial (f - (\psiq-a)h) = 0$, so that $f = (\psiq-a)h+c$, with $c \in F$. By corollary \ref{surjplus}, there exists $b \in \brigplus{,F}$ such that $(\psiq-a)(b)=c$, so that $f=(\psiq-a)(h+b)$ and (Der) is bijective. We then have, by induction on $m \geq 1$, that $\brig{}{,F} / (\psiq-q^{-1} \pi^m) = F$, generated by the class of $\partial^m(x)$.
\end{proof}

\begin{rema}
More generally, we expect that the following holds: if $\dfont$ is a $(\phi,\Gamma)$-module over $\brig{}{,K}$, the $F$-vector space $\dfont/(\psiq-1)$ is finite dimensional.
\end{rema}

\subsection{The operator $\Theta_b$}
\label{nabdiv}
The power series $F(X) = X / (\exp(X)-1)$ belongs to $\Qp \dcroc{X}$ and has a nonzero radius of convergence.
If $M$ is a Banach space with a locally $F$-analytic action of $\Gamma_K$ and $h \in \Gamma_K$ is close enough to $1$, then 
\[ \frac{\nabla}{h-1} = \frac{\nabla}{\exp(\ell(h)\nabla)-1} = \ell(h)^{-1} F(\ell(h) \nabla) \]
converges to a continuous operator on $M$. If $g \in \Gamma_K$, we then 
define
\[ \frac{\nabla}{1-g} = \frac{\nabla}{1-g^n} \cdot \frac{1-g^n}{1-g}. \]
This operator is independent of the choice of $n$ such that $g^n$ is close enough to $1$, and can be seen as an element of the locally $F$-analytic distribution algebra acting on $M$.

If $M$ is a Fr\'echet space, write $M = \projlim_i M_i$ and define operators 
$\frac{\nabla}{1-g}$ on each $M_i$ as above. These operators commute with 
the maps $M_j \to M_i$ (because $n$ can be taken large enough for both 
$M_i$ and $M_j$). This defines an operator $\frac{\nabla}{1-g}$ on $M$ 
itself. The definition of $\frac{\nabla}{1-g}$ extends to an LF space with a pro-$F$-analytic action of $\Gamma_K$.

Assume that $K$ contains $F_1$ and let $r(K)=f+\vp([K:F_1])$. For example, $p^{r(F_n)}=q^n$ if $n \geq 1$. Assume further that $K$ contains $F_{n(K)}$, so that $\chilt: \Gamma_K \to \OO_F^\times$ is injective and its image is a free $\Zp$-module of rank $d$. If $b = (b_1,\hdots,b_d)$ is a basis of $\Gamma_K$ (that is, $\Gamma_K = b_1^{\Zp} \cdots b_d^{\Zp}$), then let $\ell^*(b) = \ell(b_1) \cdots \ell(b_d)/ p^{r(K)}$ and
\[ \Theta_b =  \ell^*(b) \cdot \frac{\nabla^d}{(b_1-1) \cdots (b_d-1)}. \]

\begin{lemm}
\label{thetr}
If $K=F_n$ and $m \geq 0$ and $x \in F_{m+n}$, then 
\[ \Theta_b(x) = q^{-m-n} \cdot \Tr_{F_{m+n}/F_n} (x). \]
\end{lemm}

\begin{proof}
Since $\nabla = \lim_{k \to \infty} (b^{p^k}-1)/p^k\ell(b)$, we have 
\[ \Theta_b = \lim_{k \to \infty} \frac{1}{q^n p^{kd}} \cdot \frac{(b_1^{p^k}-1) \cdots (b_d^{p^k}-1)}{(b_1-1) \cdots (b_d-1)}. \] 
The set $\{b_1^{a_1} \cdots b_d^{a_d}\}$ with $0 \leq a_i \leq p^k-1$ runs through a set of representatives of $\Gamma_n/\Gamma_n^{p^k} = \Gamma_n / \Gamma_{n+ek}$ so that
\[ \frac{1}{q^n p^{kd}} \cdot \frac{(b_1^{p^k}-1) \cdots (b_d^{p^k}-1)}{(b_1-1) \cdots (b_d-1)} =  \frac{1}{q^n p^{kd}} \Tr_{F_{n+ek}/F_n} = \frac{1}{q^{n+ ek}} \cdot \Tr_{F_{n+ek}/F_n} . \]
The lemma follows from taking $k$ large enough so that $ek \geq m$.
\end{proof}

For $i \in \ZZ$, let $\nabla_i  = \nabla-i$.

\begin{lemm}
\label{thetbpsi}
If $b$ is a basis of $\Gamma_{F_n}$ and if $f(T) \in (\brigplus{,F})^{\psiq=0}$, then $\Theta_b(f(T)) \in (t_\pi/\phiq^n(T)) \cdot \brigplus{,F}$, and if $h \geq 2$ then $\nabla_{h-1} \circ \cdots \circ \nabla_1 \circ \Theta_b (f(T)) \in (t_\pi/\phiq^n(T))^h \cdot \brigplus{,F}$.
\end{lemm}

\begin{proof}
If $m \geq 1$, then by lemma \ref{thetr} and using repeatedly the fact (see \S \ref{nota}) that $\phiq \circ \psiq(f) = 1/q \cdot \sum_{z \in \LT[\pi]} f(T \oplus z)$,
\[ \Theta_b(f(u_{n+m})) = 1/q^{m+n} \cdot \Tr_{F_{m+n}/F_n} f(u_{m+n})  = \psiq^m (f)(u_n) = 0. \]
This proves the first claim, since an element $f(T) \in \brigplus{,F}$ is divisible by $t_\pi/\phiq^n(T)$ if and only if $f(u_{n+m})=0$ for all $m \geq 1$. The second claim follows easily.
\end{proof}

Let $D$ be a $\phiq$-module over $F$. Let $\phiq^{-n}\colon \brigplus{,F}[1/t_\pi] \otimes_F D \rightarrow F_n \dpar{t_\pi} \otimes_F D$  be the map 
\[ \phiq^{-n}\colon t_\pi^{-h} f(T) \otimes x \mapsto \pi^{nh} t_\pi^{-h} f(u_n \oplus \exp_{\LT}(t_\pi/\pi^n)) \otimes \phiq^{-n}(x). \] 
If $f(t_\pi) \in F_n \dpar{t_\pi} \otimes_F D$, let $\partial_D(f) \in F_n \otimes_F D$ denote the coefficient of $t_\pi^0$.

\begin{lemm}
\label{trpsi}
If $y \in ( \brigplus{,F}[1/t_\pi] \otimes_F D)^{\psiq=1}$ and if $m \geq n$, then
\[ q^{-m} \Tr_{F_m / F_n} \partial_D(\phiq^{-m}(y)) =
\begin{cases}
q^{-n} \partial_D(\phiq^{-n}(y)) & \text{if $n \geq 1$} \\
(1-q^{-1}\phiq^{-1})\partial_D(y) &  \text{if $n = 0$.}
\end{cases} \] 
\end{lemm}

\begin{proof}
If $y = t_\pi^{-\ell} \sum_{k=0}^{+\infty} 
a_k T^k \in \brigplus{,F}[1/t_\pi] \otimes_F D$, then (by definition of $\phiq^{-m}$)
\[ \phiq^{-m}(y) = \pi^{m \ell} t_\pi^{-\ell} \sum_{k=0}^{+\infty} 
\phiq^{-m} (a_k) (u_m \oplus \exp_{\LT}(t_\pi/\pi^m))^k, \] 
and $\psiq(y)=y$ means that:
\[ \phiq(y)(T) = \frac{1}{q} \sum_{[\pi](\omega)=0}
y(T \oplus \omega). \]
If $m \geq 2$, the conjugates of $u_m$ under $\Gal(F_m/F_{m-1})$ are the $\{\omega \oplus u_m\}_{[\pi](\omega) = 0}$ so that:
\begin{align*}
\Tr_{F_m / F_{m-1}} & \partial_D (\phiq^{-m}(y)) \\
&= \partial_D\left(\sum_{[\pi](\omega) = 0}
\pi^{m \ell} t_\pi^{-\ell} \sum_{k=0}^{+\infty} 
\phiq^{-m} (a_k) (\omega \oplus u_m \oplus \exp_{\LT}(t_\pi/\pi^m))^k
\right) \\*
&= \partial_D\left(\phiq^{-m}\left(\sum_{[\pi](\omega) = 0} y(T \oplus \omega)\right)\right) \\
&= q \partial_D(\phiq^{-(m-1)}(y)).
\end{align*}
For $m = 1$, the computation is similar, except that the 
conjugates of $u_1$ under $\Gal(F_1/F)$ are the $\omega$, 
where $[\pi](\omega) = 0$ but $\omega \neq 0$, which results in:
\[
\Tr_{F_1 / F} \partial_D(\phiq^{-1}(y))
= \partial_D\left(\phiq^{-1}\left(\sum_{\substack{[\pi](\omega) = 0 \\ \omega \neq 0}}
		y(T \oplus \omega)\right)\right)
= \partial_D(q y - \phiq^{-1}(y)). \]
\end{proof}

\subsection{Construction of extensions}
\label{constrext}

Let $\dfont$ be an $F$-analytic $(\phi,\Gamma)$-module over $\brig{}{,K}$. 
The space $\dfont^{\psiq=1}$ is a closed subspace of $\dfont$ and therefore an LF space. 
Take $K$ such that $K$ contains $F_{n(K)}$ and let $b$ be a basis of $\Gamma_K$. 

\begin{prop}
\label{constc}
If $y \in \dfont^{\psiq=1}$,  there is a unique cocycle 
$c_b(y) \in \Zan^1(\Gamma_K,\dfont^{\psiq=1})$ such that for all 
$1 \leq j \leq d$ and $k \geq 0$, we have
\[ c_b(y)(b_j^k) = \ell^*(b) \cdot \frac{b_j^k-1}{b_j-1} \cdot \frac{\nabla^{d-1}}{\prod_{i \neq j} (b_i-1)} (y). \]
We then have $c_b(y)'(1) = \Theta_b(y)$.
\end{prop}

\begin{proof}
There is obviously one and only one continuous cocycle satisfying the 
conditions of the proposition. It is $\Qp$-analytic, and in order to prove 
that it is $F$-analytic, we need to check that the directional derivatives 
are independent of $j$. We have 
\[
\lim_{k \to 0} \frac{c_b(y)(b_j^k)}{\ell(b_j^k)}
= \ell^*(b) \cdot \frac{\nabla^d}{\prod_i (b_i-1)} (y)
= \Theta_b(y),
\]
which is indeed independent of $j$, and thus $c_b(y)'(1) = \Theta_b(y)$.
\end{proof}

\begin{lemm}
\label{cbcoresone}
If $n \geq n(K)$ and $L=K_n$ and $M=K_{n+e}$ and $b$ is a basis of $\Gamma_L$, then $b^p$ is a 
basis of $\Gamma_M$ and $\cor_{M/L} c_{b^p}(y) = c_b(y)$.
\end{lemm}

\begin{proof}
The Lubin-Tate character maps $\Gamma_L$ to $1 + \pi^n \OO_F$, and $\Gamma_M = \Gamma_L^p$ because $(1 + \pi^n \OO_F)^p = 1 + \pi^{n+e} \OO_F$. Since $\{b_1^{k_1} \cdots b_d^{k_d}\}$ with $0 \leq k_i \leq p-1$ is a set of 
representatives for $\Gamma_L/\Gamma_M$, and since $[M:L]=q^e=p^d$, the 
explicit formula for the corestriction (definition \ref{defcor}) implies (here and elsewhere $\lceil x \rceil$ is the smallest integer $\geq x$)
\begin{align*}
\cor_{M/L} & (c_{b^p}(y)) (b_j^k) \\
&= \sum_{0 \leq k_1, \dots, k_d \leq p-1} b_1^{k_1} \dots b_d^{k_d} \cdot
    \ell^*(b^p) \cdot \frac{b_j^{p \left\lceil\frac{k-k_j}{p}\right\rceil}-1}{b_j^p-1}
    \cdot \frac{\nabla^{d-1}}{\prod_{i \neq j} (b_i^p-1)} (y) \\
&= \ell^*(b) \left(\sum_{k_j=0}^{p-1} b_j^{k_j}
    \frac{b_j^{p \left\lceil\frac{k-k_j}{p}\right\rceil}-1}{b_j^p-1}\right)
    \cdot \left(\prod_{i \neq j} \frac{b_i^p-1}{b_i-1}\right)
    \cdot \frac{\nabla^{d-1}}{\prod_{i \neq j} (b_i^p-1)} (y) \\
&= \ell^*(b) \frac{b_j^k-1}{b_j-1} \cdot \frac{\nabla^{d-1}}{\prod_{i \neq j} (b_i-1)} (y) \\
&= c_b(y)(b_j^k).
\end{align*}

This proves the lemma.
\end{proof}

\begin{lemm}
\label{cbycay}
If $a$ and $b$ are two bases of $\Gamma_K$, then $c_a(y)$ and $c_b(y)$ are cohomologous.
\end{lemm}

\begin{proof}
If $\alpha_1,\hdots,\alpha_d$ and $\beta_1,\hdots,\beta_d$ are in $F^\times$, 
the Laurent series
\[ \frac{\alpha_1 \cdots \alpha_d \cdot T^{d-1}}{(\exp(\alpha_1 T)-1) \cdots (\exp(\alpha_d T)-1)}
    - \frac{\beta_1 \cdots \beta_d \cdot T^{d-1}}{(\exp(\beta_1 T)-1) \cdots (\exp(\beta_d T)-1)} \]
is the difference of two Laurent series, each having a simple pole at $0$ 
with equal residues, and therefore belongs to $F\dcroc{T}$. Let $a$ and $b$ 
be two bases of $\Gamma_K$ and take $y \in \dfont^{\psiq=1}$. 

Let $N$ be a $\Gamma_K$-stable Fr\'echet subspace of $\dfont$ that contains $y$ and 
write $N = \varprojlim M_j$.
Since $M = M_j$ is $F$-analytic, we have $g = \exp(\ell(g) \nabla)$ on $M$ 
for $g$ in some open subgroup of $\Gamma_K$. Let $k \gg 0$ be large enough such that 
$a_i^{p^k}$ and $b_i^{p^k}$ are in this subgroup, and let 
$\alpha_i = p^k \ell(a_i)$ and $\beta_i = p^k \ell(b_i)$. Taking $k$ 
large enough (depending on $M$), we can assume moreover that the power 
series $T/(\exp(T) - 1)$ applied to the operators $\alpha_i \nabla$ 
and $\beta_i \nabla$ converges on $M$. The element
\begin{multline*}
w = \bigg( \frac{\alpha_1 \cdots \alpha_d \cdot \nabla^{d-1}}{(\exp(\alpha_1 \nabla)-1) \cdots (\exp(\alpha_d \nabla)-1)}
 - \frac{\beta_1 \cdots \beta_d \cdot \nabla^{d-1}}{(\exp(\beta_1 \nabla)-1) \cdots (\exp(\beta_d \nabla)-1)} \bigg) (y)
\end{multline*}
of $M$ is well defined. By proposition~\ref{constc}, we have
\[
c_{a^{p^k}}(y)'(1) - c_{b^{p^k}}(y)'(1)
= \Theta_{a^{p^k}}(y) - \Theta_{b^{p^k}}(y)
= p^{-r(L)} \nabla(w)
\]
where $L$ is the extension of $K$ such that $\Gamma_L = \Gamma_K^{p^k}$. 
Thus, for $g$ close enough to $1$, we have
$c_{a^{p^k}}(y)(g) - c_{b^{p^k}}(y)(g) = (g-1)(p^{-r(L)} w)$. 
Lemma~\ref{cbcoresone} now implies by corestricting that this holds for 
all $g$, and, by corestricting again, that $c_a(y)$ and $c_b(y)$ are cohomologous in $M$. By varying $M$, we get the same result in $N$, which implies the proposition.
\end{proof}

\begin{lemm}
\label{cbcores}
If $L/K$ is a finite extension contained in $K_\infty$, and if $b$ is a basis of $\Gamma_K$ and 
$a$ is a basis of $\Gamma_L$, then $\cor_{L/K} c_a(y) = c_b(y)$.
\end{lemm}

\begin{proof}
The groups $\Gamma_K$ and $\Gamma_L$ are both free $\Zp$-modules of rank~$d$, 
so that by the elementary divisors theorem, we can 
change the bases $a$ and $b$ in such a way that there exists $e_1,\hdots,e_d$ 
with $a_i = b_i^{p^{e_i}}$.

Since $\{b_1^{k_1} \cdots b_d^{k_d}\}$ with $0 \leq k_i \leq p^{e_i}-1$ is a 
set of representatives for $\Gamma_K/\Gamma_L$, and since 
$[L:K]=p^{e_1+\cdots+e_d}$, the explicit formula for the corestriction implies

\begin{align*}
\cor_{L/K} & (c_a(y))(b_j^k) \\
&= \sum_{\substack{0 \leq k_1 \leq p^{e_1}-1 \\ \dots \\ 0 \leq k_d \leq p^{e_d}-1}}
    b_1^{k_1} \dots b_d^{k_d} \cdot \ell^*(a) \cdot 
    \frac{a_j^{\left\lceil\frac{k-k_j}{p^{e_j}}\right\rceil}-1}{a_j-1} \cdot 
    \frac{\nabla^{d-1}}{\prod_{i \neq j} (a_i-1)} (y) \\
&= \ell^*(b) \cdot \left(\sum_{k_j=0}^{p^{e_j}-1} 
    \frac{a_j^{\left\lceil\frac{k-k_j}{p^{e_j}}\right\rceil}-1}{a_j-1}\right) \cdot 
    \left(\prod_{i \neq j} \frac{a_i-1}{b_i-1}\right) \cdot
    \frac{\nabla^{d-1}}{\prod_{i \neq j} (a_i-1)} (y) \\
&= \ell^*(b) \cdot \frac{b_j^k-1}{b_j-1} \cdot
    \frac{\nabla^{d-1}}{\prod_{i \neq j} (b_i-1)} (y) \\
&= c_b(y)(b_j^k).
\end{align*}
\end{proof}

\begin{defi}
\label{defhun}
Let $h^1_{K,V} : \drig{}(V)^{\psiq=1} \to \Han^1(K,V)$ denote the map obtained by composing $y \mapsto \overline{c}_b(y)$ with $\Han^1(\Gamma_K,\drig{}(V)^{\psiq=1}) \to  \Han^1(\Gamma_K \times \Psi, \drig{}(V))$ (theorem \ref{infrespsi}) and with $ \Han^1(\Gamma_K \times \Psi, \drig{}(V)) \simeq \Han^1(K,V)$ (proposition \ref{h1anextv} and corollary \ref{cohphipsi}).
\end{defi}

\begin{prop}
\label{h1corc}
We have $\cor_{M/L} \circ h^1_{M,V} = h^1_{L,V}$ if $M/L$ is a finite extension contained in $K_\infty/K_{n(K)}$. In particular, $\cor_{K_{n+1}/K_n} \circ h^1_{K_{n+1},V} = h^1_{K_n,V}$ if $n \geq n(K)$.
\end{prop}

\begin{proof}
This follows from the definition and from lemma \ref{cbcores} above.
\end{proof}

\begin{rema}
\label{defhkall}
Proposition \ref{h1corc} allows us to extend the definition of $h^1_{K,V}$ to all $K$, without assuming that $K$ contains $F_{n(K)}$, by corestricting.
\end{rema}

Some of the constructions of this section are summarized in the following theorem. Recall (see \S 3 of \cite{PGMLAV}) that there is a ring $\btrig{}{}$ that contains $\brig{}{,F}$, is equipped with a Frobenius map $\phiq$ and an action of $G_F$ and such that $V = (\btrig{}{} \otimes_{\brig{}{,F}} \drig{}(V))^{\phiq=1}$.

\begin{theo}
\label{hunkv}
If $y \in \drig{}(V)^{\psiq=1}$ and $K$ contains $K_{n(K)}$ and $b$ is a basis of $\Gamma_K$, then 
\begin{enumerate}
\item there is a unique $c_b(y) \in \Zan^1(\Gamma_K,\drig{}(V)^{\psiq=1})$ such that for $k \in \Zp$,
\[ c_b(y)(b_j^k) = \ell^*(b) \cdot \frac{b_j^k-1}{b_j-1} \cdot \frac{\nabla^{d-1}}{\prod_{i \neq j} (b_i-1)} (y); \]
\item there is a unique $m_c \in \drig{}(V)^{\psiq=0}$ such that $(\phiq -1)c_b(y)(g) = (g-1) m_c$ for all $g \in \Gamma_K$; 
\item the $(\phi,\Gamma)$-module corresponding to this extension has a basis in which
\[ \Mat(g) = \pmat{\ast & c_b(y)(g) \\ 0 & 1}\text{ if $g \in \Gamma_K$,} \qquad \text{and}\qquad  \Mat(\phiq) = \pmat{\ast & m_c \\ 0 & 1}; \]
\item if $z \in \btrig{}{} \otimes_F V$ is such that $(\phiq-1)z = m_c$, then the cocycle
\[ g \mapsto c_b(y)(g) - (g-1)z \]
defined on $G_K$ has values in $V$ and represents $h^1_{K,V}(y)$ in $\Han^1(K,V)$.
\end{enumerate}
\end{theo}

\begin{proof}
Items (1), (2) and (3) are reformulations of the constructions of this chapter. Let us prove (4). Let us write the $(\phi,\Gamma)$-module corresponding to the extension in (3) as $\dfont' = \drig{}(V) \oplus \brig{}{,F} \cdot e$. It is an \'etale $(\phi,\Gamma)$-module that comes from the $p$-adic representation $V' = (\btrig{}{} \otimes_{\brig{}{,F}} \dfont')^{\phiq=1}$. We have $V' = V \oplus F \cdot (e-z)$ as $F$-vector spaces since $\phiq(e-z)=e-z$. If $g \in G_K$, then 
\[ g(e-z)= e + c_b(y)(g) - g(z)= e-z + c_b(y)(g) -(g-1)z. \] This proves (4).
\end{proof}

Let $F=\Qp$ and $\pi=p=q$, and let $V$ be a representation of $G_K$. In \S II.1 of \cite{CC99}, Cherbonnier and Colmez define a map $\Log^*_{V^*(1)} : \ddag{}(V)^{\psi=1} \to \Hiw^1(K,V)$, which is an isomorphism (theorem II.1.3 and proposition III.3.2 of \cite{CC99}).

\begin{prop}\label{compcc}
If $F=\Qp$ and $\pi=p$, then the map 
\[ \ddag{}(V)^{\psi=1} \to \drig{}(V)^{\psi=1} \xrightarrow{\{h^1_{K_n,V}\}_{n \geq 1}} \projlim_n \Han^1(K_n,V) \to \projlim_n \cH^1(K_n,V) \] 
coincides with the map $\Log^*_{V^*(1)} : \ddag{}(V)^{\psi=1} \to \Hiw^1(K,V) \subset \projlim_n \cH^1(K_n,V)$.
\end{prop}

\begin{proof}
The map $\Log^*_{V^*(1)}$ is contructed by mapping $x \in \ddag{}(V)^{\psi=1}$ to the sequence  
$(\dots, \iota_{\psi,n}(x), \dots) \in \projlim_n \cH^1(K_n,V)$ (see theorem~II.1.3 in \cite{CC99} 
and the paragraph preceding it), where
\[ \iota_{\psi,n}(x) = \left[ \sigma \mapsto \ell_{K_n}(\gamma_n) \left(
\frac{\sigma-1}{\gamma_n-1} x - (\sigma-1) b
\right) \right] \]
on $G_{K_n}$ and where (see proposition I.4.1, lemma I.5.2 and lemma I.5.5 of ibid.)
\begin{enumerate}
\item $\gamma_n = \gamma_1^{[K_n:K_1]}$ and $\gamma_1$ is a fixed generator of $\Gamma_{K_1}$;
\item $\ell_{K_n}(\gamma_n) = \frac{\log \chi(\gamma_n)}{p^{r(K_n)}}$ where $r(K_n)$ is the integer 
such that $\log \chi(\Gamma_{K_n}) = p^{r(K_n)} \Zp$;
\item $b \in \btdag{} \otimes_{\Qp} V$ is such that $(\phi-1) b = a$ and $a \in \ddag{}(V)^{\psi=1}$ 
is such that $(\gamma_n-1) a = (\phi-1) x$ (using the fact that $\gamma_n-1$ is bijective on 
$\ddag{}(V)^{\psi=0}$).
\end{enumerate}
The theorem follows from comparing this with the explicit formula of theorem \ref{hunkv}.
\end{proof}

\section{Explicit formulas for crystalline representations}
\label{formsec}

In this chapter, we explain how the constructions of the previous chapter are related to $p$-adic Hodge theory, via Bloch and Kato's exponential maps. Let $\bdr$ be Fontaine's ring of periods \cite{FPP} and let $\bmax{F}^+$ be the subring of $\bdr^+$ that is constructed in \S 8.5 of \cite{C02} (recall that $\bmax{F}^+ = F \otimes_{F_0} \bfont_{\max}^+$ where $F_0 = F \cap \Qp^{\unr}$ and $\bfont_{\max}^+$ is a ring that is similar to Fontaine's $\bfont_{\mathrm{cris}}$).

We assume throughout this chapter that $K=F$ and that the representation $V$ is crystalline and $F$-analytic.

\subsection{Crystalline $F$-analytic representations}
\label{kremind}

If $V$ is an $F$-analytic crystalline representation of $G_F$, let $\dcris(V) = (\bmax{F} \otimes_F V)^{G_F}$ (this is the ``component at identity'' of the usual $\dcris$). By corollary 3.3.8 of \cite{KR09}, $F$-analytic crystalline representations of $G_F$ are overconvergent. Moreover, if $\mathcal{M}(D) \subset \brigplus{,F}[1/t_\pi] \otimes_F D$ is the object constructed in \S 2.2 of ibid., then by \S 2.4 of ibid., $\mathcal{M}(\dcris(V))$ contains a basis of $\ddag{}(V)$ and $\drig{}(V) = \brig{}{,F} \otimes_{\brigplus{,F}} \mathcal{M}(\dcris(V))$. This implies that $\drig{}(V) \subset \brig{}{,F}[1/t_\pi] \otimes_F \dcris(V)$. 

\begin{theo}
\label{psiplus}
We have $\drig{}(V)^{\psiq=1} \subset \brigplus{,F}[1/t_\pi] \otimes_F \dcris(V)$.
\end{theo}

\begin{proof}
Take $h \geq 0$ such that the slopes of $\pi^{-h} \phiq$ on $\dcris(V)$ are $\leq -d$. Let $E$ be an extension of $F$ such that $E$ contains the eigenvalues of $\phiq$ on $\dcris(V)$. We show that $\drig{}(V)^{\psiq=1} \subset t_\pi^{-h} E \otimes_F \brigplus{,F} \otimes_F \dcris(V)$. Let $e_1,\hdots,e_n$ be a basis of $t_\pi^{-h} E \otimes_F \dcris(V)$ in which the matrix $(p_{i,j})$ of $\phiq$ is upper triangular. If $y = \sum_{i=1}^d y_i \otimes \phiq(e_i)$ with $y_i \in E \otimes_F \brig{}{,F}$, then $\psiq(y)=y$ if and only if $\psiq(y_k) = p_{k,k} y_k + \sum_{j > k} p_{k,j} y_j$ for all $k$. The theorem follows from applying lemma \ref{psireg} below to $k=n,n-1,\hdots,1$.
\end{proof}

\begin{lemm}
\label{psireg}
Take $y \in E \otimes_F \brig{}{,F}$ and $\alpha \in F$ such that $\val_\pi(\alpha) \leq -d$. If $\psiq(y)- \alpha y \in E \otimes_F \brigplus{,F}$, then $y \in E \otimes_F \brigplus{,F}$.
\end{lemm}

\begin{proof}
This is lemma 5.4 of \cite{FX13}.
\end{proof}

\subsection{Bloch-Kato's exponentials for analytic representations}
\label{fanexp}
We now recall the definition of Bloch-Kato's exponential map and its dual, and give a similar definition for $F$-analytic representations.

\begin{lemm}
\label{fundfan}
We have an exact sequence 
\[ 0 \to F \to (\bmax{F}^+[1/t_\pi])^{\phiq=1} \to \bdr/\bdr^+ \to 0. \]
\end{lemm}

\begin{proof}
This is lemma~9.25 of~\cite{C02}.
\end{proof}

If $V$ is a de Rham $F$-linear representation of $G_K$, we can $\otimes_F$ the above sequence with $V$ and we get a connecting homomorphism $\exp_{K,V} : (\bdr \otimes_F V)^{G_K} \to \cH^1(K,V)$. Recall that if $W$ is an $F$-vector space, there is a natural injective map $W \otimes_F V \to W \otimes_{\Qp} V$.

\begin{lemm}
\label{expcompat}
If $V$ is $F$-analytic, the map $\exp_{K,V} : (\bdr \otimes_F V)^{G_K} \to \cH^1(K,V)$ defined above coincides with Bloch-Kato's exponential via the inclusion $(\bdr \otimes_F V)^{G_K} \subset (\bdr \otimes_{\Qp} V)^{G_K}$, and its image is in $\Han^1(K,V)$.
\end{lemm}

\begin{proof}
Bloch and Kato's exponential is defined as follows (definition 3.10 of \cite{BK90}): if $\phi_p$ denotes the Frobenius map that lifts $x \mapsto x^p$ and if $x \in (\bdr \otimes_{\Qp} V)^{G_K}$, there exists $\tilde{x} \in \bmax{\Qp}^{\phi_p=1} \otimes_{\Qp} V$ such that $\tilde{x} - x \in \bdr^+ \otimes_{\Qp} V$, and $\exp(x)$ is represented by the cocyle $g \mapsto (g-1)\tilde{x}$.

Lemma \ref{fundfan} says that we can lift $x \in (\bdr \otimes_F V)^{G_K}$ to some $\tilde{x} \in (\bmax{F}^+[1/t_\pi])^{\phiq=1} \otimes_F V$ such that $\tilde{x}-x \in \bdr^+ \otimes_F V \subset \bdr^+ \otimes_{\Qp} V$. In addition, $\bmax{\Qp}^{\phiq=1} = F_0 \otimes_{\Qp} \bmax{\Qp}^{\phi_p=1}$ 
(see lemma 1.1.11 of \cite{LB8})
so that $(\bmax{F}^+[1/t_\pi])^{\phiq=1} \subset F \otimes_{\Qp} \bmax{\Qp}^{\phi_p=1}$. We can therefore view $\tilde{x}$ as an element of $\bmax{\Qp}^{\phi_p=1} \otimes_{\Qp} V$, and $\exp_{K,V}(x)= [g \mapsto (g-1)\tilde{x}] = \exp(x)$.

The construction of $\exp_{K,V}(x)$ shows that the cocycle $\exp_{K,V}(x)$ is de Rham. At each embedding $\tau \neq \Id$ of $F$, the extension of $F$ by $V$ given by $\exp_{K,V}(x)$ is therefore Hodge-Tate with weights $0$. This finishes the proof of the lemma.
\end{proof}

Recall the following theorem of Kato (see \S II.1 of \cite{K93}).

\begin{theo}
\label{katofor}
If $V$ is a de Rham representation, the map from $(\bdr \otimes_{\Qp} V)^{G_K}$ to
$\cH^1(K,\bdr \otimes_{\Qp} V)$ defined by 
$x \mapsto \left[ g \mapsto \log(\chi_{\cyc}(\overline{g}))x \right]$  
is an isomorphism, and the dual
exponential map $\exp^*_{K,V^*(1)} : \cH^1(K,V) \to (\bdr \otimes_{\Qp} V)^{G_K}$
is equal to the composition of the map $\cH^1(K,V) \to
\cH^1(K,\bdr \otimes_{\Qp} V)$ with the inverse of this isomorphism.
\end{theo}

Concretely, if $c \in \cZ^1(K,\bdr \otimes_{\Qp} V)$ is some cocycle, there exists $w \in \bdr \otimes_{\Qp} V$ such that $c(g) = \log(\chi_{\cyc}(\overline{g})) \cdot \exp^*_{K,V^*(1)}(c) + (g-1)(w)$.

\begin{coro}
\label{fanduexp}
If $c \in \cZ^1(K,\bdr \otimes_F V)$, and if there exist $x \in (\bdr \otimes_F V)^{G_K}$ and $w \in \bdr \otimes_F V$ such that $c(g) = \ell(\overline{g}) \cdot x + (g-1)(w)$, then $\exp^*_{K,V^*(1)}(c) =x$.
\end{coro}

\begin{proof}
This follows from theorem \ref{katofor} and from the fact that $g \mapsto \log(\chilt(\overline{g})/\chi_{\cyc}(\overline{g}))$ is $\bdr$-admissible, since $t_\pi/t \in (\bdr^+)^\times$ so that $\log(t_\pi/t) \in \bdr^+$ is well-defined.
\end{proof}

\subsection{Interpolating exponentials and their duals}
\label{interexp}

Let $V$ be an $F$-analytic crystalline representation. By theorem \ref{psiplus}, we have $\drig{}(V)^{\psiq=1} \subset \brigplus{,F}[1/t_\pi] \otimes_F \dcris(V)$. Let $\partial_V$ denote the map $\partial_D$ of \S \ref{nabdiv} for $D=\dcris(V)$.

\begin{theo}
\label{interdual}
If $y \in \drig{}(V)^{\psiq=1}$, then
\[ \exp^*_{F_n, V^*(1)}(h^1_{F_n,V}(y)) = 
\begin{cases}
q^{-n} \partial_V(\phiq^{-n}(y)) & \text{if $n \geq 1$} \\
(1-q^{-1}\phiq^{-1})\partial_V(y) &  \text{if $n = 0$.}
\end{cases} \]
\end{theo}

\begin{proof}
Since the diagram 
\[ \begin{CD} 
\cH^1(F_{n+1},V) @>{\exp^*_{F_{n+1},V^*(1)}}>> F_{n+1} \otimes_F
\dcris(V) \\ 
@V{\cor_{F_{n+1}/F_n}}VV @V{\Tr_{F_{n+1}/F_n}}VV \\
\cH^1(F_n,V) @>{\exp^*_{F_n,V^*(1)}}>> F_n \otimes_F \dcris(V)
\end{CD} \]
is commutative, we only need to prove the theorem when $n \geq n(F)$ 
by lemma \ref{trpsi} and proposition \ref{h1corc}. By theorem \ref{hunkv}, we have
\[ h^1_{F_n,V}(y)(b_j^k) = \ell^*(b) \cdot \frac{b_j^k-1}{b_j-1} \cdot \frac{\nabla^{d-1}}{\prod_{i \neq j} (b_i-1)} (y) - (b_j^k-1)z, \]
with $z \in \btrig{}{} \otimes_F V$ so that if $m \gg 0$, then $\phiq^{-m}(z) \in \bdr^+ \otimes_F V$ (see \S 3 of \cite{PGMLAV} and \S 2.2 of \cite{LB2}). Moreover, $\phiq^{-m}(y) \in F_m \dpar{t_\pi} \otimes_F \dcris(V)$. Let $W = \{ w \in F_m \dpar{t_\pi} \otimes_F \dcris(V)$ such that $\partial_V(w) = 0\}$. The operator $\nabla$ is bijective on $W$, and $F_m \dpar{t_\pi} \otimes_F \dcris(V)$ injects into $\bdr \otimes_F V$, hence there exists $u \in \bdr \otimes_F V$ such that
\begin{align*} h^1_{F_n,V}(y)(b_j^k) & =\ell^*(b) \cdot \frac{b_j^k-1}{b_j-1} \cdot \frac{\nabla^{d-1}}{\prod_{i \neq j} (b_i-1)} (\partial_V(\phiq^{-m}(y))) - (b_j^k-1)u \\
& = \ell(b_j^k) \cdot  \Theta_b (\partial_V(\phiq^{-m}(y))) - (b_j^k-1)u \\
& = \ell(b_j^k) \cdot q^{-n} \partial_V (\phiq^{-n}(y))) - (b_j^k-1)u,
\end{align*} 
by lemmas \ref{thetr} and \ref{trpsi}. This proves the theorem by corollary \ref{fanduexp}.
\end{proof}

We now give explicit formulas for $\exp_{F_n,V}$. Take $h \geq 0$ such that $\Fil^{-h} \dcris(V) = \dcris(V)$, so that  $t_\pi^h (\brigplus{,F} \otimes_F \dcris(V)) \subset \drig{}(V)$ (in the notation of \S 2.2 of \cite{KR09}, we have $t_\pi^h (\brigplus{,F} \otimes_F \dcris(V)) \subset \mathcal{M}(\dcris(V))$). In particular, if $y \in (\brigplus{,F} \otimes_F \dcris(V))^{\psiq=1}$, then $\nabla_{h-1} \circ  \cdots \circ \nabla_0 (y) \in \drig{}(V)^{\psiq=1}$.

\begin{theo}\label{expbk}
If $y \in (\brigplus{,F} \otimes_F \dcris(V))^{\psiq=1}$, then
\begin{multline*}
h^1_{F_n,V}
(\nabla_{h-1} \circ  \cdots \circ \nabla_0 (y)) = 
(-1)^{h-1} (h-1)!
\begin{cases}
\exp_{F_n,V}(q^{-n} \partial_V(\phiq^{-n}(y))) & \text{if $n \geq 1$} \\
\exp_{F,V}((1-q^{-1}\phiq^{-1})\partial_V(y)) & \text{if $n=0$.}
\end{cases} 
\end{multline*}
\end{theo}

\begin{proof}
Since the diagram 
\[ \begin{CD} 
F_{n+1} \otimes_F \dcris(V) @>{\exp_{F_{n+1},V}}>> \cH^1(F_{n+1},V) \\
@V{\Tr_{F_{n+1}/F_n}}VV @V{\cor_{F_{n+1}/F_n}}VV \\
F_n \otimes_F \dcris(V) @>{\exp_{F_n,V}}>> \cH^1(F_n,V)
\end{CD} \]
is commutative, we only need to prove the theorem when $n \geq n(F)$ by lemma \ref{trpsi} and proposition \ref{h1corc}. By theorem \ref{hunkv}, we have
\begin{multline*} 
h^1_{F_n,V}(\nabla_{h-1} \circ  \cdots \circ \nabla_0 (y))(b_j^k) \\ 
= \ell^*(b) \cdot \frac{b_j^k-1}{b_j-1} \cdot \frac{\nabla^{d-1}}{\prod_{i \neq j} (b_i-1)} (\nabla_{h-1} \circ  \cdots \circ \nabla_0 (y)) - (b_j^k-1)z \\ 
= (b_j^k-1) \cdot  (\nabla_{h-1} \circ  \cdots \circ \nabla_1 \circ \Theta_b)(y) - (b_j^k-1)z,
\end{multline*}
so that $h^1_{F_n,V}(\nabla_{h-1} \circ  \cdots \circ \nabla_0 (y))(g) = (g-1) (\nabla_{h-1} \circ  \cdots \circ \nabla_1 \circ \Theta_b)(y) - (g-1)z$ if $g \in \Gamma_K$. By lemma \ref{thetbpsi}, we have 
\begin{multline*}
(\nabla_{h-1} \circ  \cdots \circ \nabla_1 \circ \Theta_b)((\phiq-1)y)  \in (t_\pi/\phiq^n(T))^h (\brigplus{,F} \otimes_F \dcris(V))^{\psiq=0} \subset  \drig{}(V)^{\psiq=0},
\end{multline*} 
so that (in the notation of theorem \ref{hunkv}) $m_c = (\nabla_{h-1} \circ  \cdots \circ \nabla_1 \circ \Theta_b)((\phiq-1)y)$. Since $(\phiq-1)z=m_c$, we have $(\phiq-1)((\nabla_{h-1} \circ  \cdots \circ \nabla_1 \circ \Theta_b)(y)-z) = 0$, and therefore
\[ (\nabla_{h-1} \circ  \cdots \circ \nabla_1 \circ \Theta_b)(y)-z \in (\btrig{}{}[1/t_\pi])^{\phiq=1} \otimes_F V \] 
The ring $\btrig{}{}$ contains $\bmax{F}^+$ and the inclusion $(\bmax{F}^+[1/t_\pi])^{\phiq=1} \subset (\btrig{}{}[1/t_\pi])^{\phiq=1}$ is an equality (proposition 3.2 of \cite{LB2}). This implies that
\[  (\nabla_{h-1} \circ  \cdots \circ \nabla_1 \circ \Theta_b)(y)-z \subset (\bmax{F}^+[1/t_\pi])^{\phiq=1} \otimes_F V. 
\]
Moreover, we have $z \in \btrig{}{} \otimes_F V$ so that if $m \gg 0$, then $\phiq^{-m}(z) \in \bdr^+ \otimes_F V$. In addition, $\phiq^{-m}(y)$ belongs to $F_m \dcroc{t_\pi} \otimes_F \dcris(V)$, so that $\phiq^{-m}(y)-\partial_V (\phiq^{-m}(y))$ belongs to $t_\pi F_m \dcroc{t_\pi} \otimes_F \dcris(V)$ and therefore
\begin{align*} 
(\nabla_{h-1} \circ  \cdots \circ \nabla_1 \circ \Theta_b)\left(\phiq^{-m}(y)-\partial_V (\phiq^{-m}(y))\right) & \in t_\pi^h F_m \dcroc{t_\pi} \otimes_F \dcris(V) \\
& \subset \bdr^+ \otimes_F V. 
\end{align*}
We can hence write 
\[ h^1_{F_n,V}(\nabla_{h-1} \circ  \cdots \circ \nabla_0 (y))(g) = (g-1) ( \nabla_{h-1} \circ  \cdots \circ \nabla_1 \circ \Theta_b \circ\partial_V (\phiq^{-m}(y)) - (g-1)u, \]
with $u \in \bdr^+ \otimes_F V$. The theorem now follows from the fact that 
\[ \Theta_b \circ\partial_V (\phiq^{-m}(y)) = q^{-n} \partial_V (\phiq^{-n}(y)) \in F_n \otimes_F \dcris(V) \] 
by lemmas \ref{thetbpsi} and \ref{trpsi}, that $\nabla_{h-1} \circ  \cdots \circ \nabla_1 = (-1)^{h-1}(h-1)!$ on $F_n \otimes_F \dcris(V)$, and from the reminders given in \S \ref{fanexp}, in particular the fact that $\exp_{K,V}$ is the connecting homomorphism when tensoring the exact sequence of lemma \ref{fundfan} with $V$ and taking Galois invariants.
\end{proof}

\subsection{Kummer theory and the representation $F(\chilt)$}
\label{expkum}

Throughout this section, $V=F(\chilt)$. Let $L \subset \Qpbar$ be an extension of $K$. The Kummer map $\delta : \LT(\MM_L) \to \cH^1(L,V)$ is defined as follows. Choose a generator $u=(u_k)_{k \geq 0}$ of $T_\pi \LT = \projlim_k \LT[\pi^k]$. If $x \in \LT(\MM_L)$, let $x_k \in \LT(\MM_{\Qpbar})$ be such that $[\pi^k](x_k)=x$. If $g \in G_L$, then $g(x_k) - x_k \in \LT[\pi^k]$ so that we can write $g(x_k)-x_k = [c_k(g)](u_k)$ for some $c_k(g) \in \OO_F / \pi^k$. If $c(g)=(c_k(g))_{k \geq 0} \in \OO_F$ then $\delta(x) = [g \mapsto c(g)] \in \cH^1(L,V)$. 

If $x \in \LT(\MM_L)$, let $\TT_{L/K}$ be defined by $\TT_{L/K}(x) = \sum_{g \in \Gal(L/K)}^{\LT} g(x)$ where the superscript $\LT$ means that the summation is carried out using the Lubin-Tate addition. If $F=\Qp$ and $\LT=\Gm$, we recover the classical Kummer map, and $\TT_{L/K}(x) = \Nm_{L/K}(1+x)-1$.

\begin{lemm}\label{kumcor}
We have the following commutative diagram:
\[ \begin{CD} \LT(\MM_{K_{n+1}}) @>{\delta}>> \cH^1(K_{n+1},V) \\
@V{\TT_{K_{n+1}/K_n}}VV @VV{\cor_{K_{n+1}/K_n}}V \\
\LT(\MM_{K_n}) @>{\delta}>> \cH^1(K_n,V). \end{CD} \]
\end{lemm}

\begin{proof}
This is a straightforward consequence of the explicit description of the corestriction map.
\end{proof}

Recall that $\phiq \circ \psiq(f) = \frac{1}{q} \sum_{\omega \in \LT[\pi]} f(T \oplus \omega)$, so that for $n \geq 1$:
\[ \psiq(f)(u_n) = \frac{1}{q} \sum_{\omega \in \LT[\pi]} f(u_{n+1} \oplus \omega) =  \frac{1}{q} \Tr_{F_{n+1}/F_n} f(u_{n+1}). \]
In particular, if $f(T) \in \brigplus{,F}$ is such that $\psiq (f(T)) = 1/\pi \cdot f(T)$ and $y_n = f(u_n)$, then $\Tr_{F_{n+1}/F_n}(y_{n+1}) = q/\pi \cdot y_n$.

\begin{prop}
\label{interlog}
Assume that $F \neq \Qp$. If $\{y_n\}_{n \geq 1}$ is a sequence with $y_n \in F_n$ and $\Tr_{F_{n+1} / F_n}(y_{n+1}) = q/\pi \cdot y_n$,  there exists $f(T) \in \brigplus{,F}$ such that $\psiq (f(T)) = 1/\pi \cdot f(T)$ and $y_n = f(u_n)$ for all $n \geq 1$.
\end{prop}

\begin{proof}
By \cite{L62}, there exists a power series $g(T) \in \brigplus{,F}$ such that 
$g(u_n) = y_n$ for all $n \geq 1$. We also have
\[ \psiq g (0) = \frac{1}{q} g(0) + \frac{1}{q} \Tr_{F_1/F_0} g(u_1), \]
and since $q \neq \pi$ (because $F \neq \Qp$),  we can choose $g(0)$ such that
\[ \frac{1}{\pi} g (0) = \frac{1}{q} g(0) + \frac{1}{q} \Tr_{F_1/F_0} y_1. \]
This implies that $(\psiq(g)-1/\pi \cdot g)(u_n)=0$ for all $n \geq 0$, 
so that $\psiq(g)-1/\pi \cdot g \in t_\pi \cdot \brigplus{,F}$. It is therefore 
enough to prove that $\psiq-1/\pi  : t_\pi \cdot \brigplus{,F} \to t_\pi \cdot \brigplus{,F}$ 
is onto. Since $\psiq(t_\pi f) = 1/\pi \cdot t_\pi \psiq(f)$, this amounts to 
proving that $\psiq-1 : \brigplus{,F} \to \brigplus{,F}$ is onto, which follows from corollary \ref{surjplus}.
\end{proof}

\begin{defi}
\label{defs}
Let $S$ denote the set of sequences $\{x_n\}_{n \geq 1}$ with $x_n \in \MM_{F_n}$ and $\Tr^{\LT}_{F_{n+1}/F_n}(x_{n+1}) = [q/\pi](x_n)$ for $n \geq 1$.
\end{defi}

The following proposition says that if $F \neq \Qp$, then $S$ is quite large: for any $k \geq 1$, the ``$k$-th component'' map $F \otimes_{\OO_F} S \to F_k$ is surjective (if $F=\Qp$, there are restrictions on ``universal norms'').

\begin{prop}
\label{notzero}
Assume that $F \neq \Qp$. If $z \in \MM_{F_k}$,  there exists $\ell \geq 0$ and $x \in S$ such that $x_k = [\pi^\ell](z)$.
\end{prop}

\begin{proof}
We claim that $\Tr_{F_{n+1}/F_n}(\OO_{F_{n+1}}) = \pi \OO_{F_n}$. Indeed, 
let $\mathcal{D}$ denote the different. We have (see for instance 
proposition~7.11 of~\cite{IK})
\[ \vp(\mathcal{D}_{F_{n+1}/F_n}) = \frac{1}{e} \left(n+1-\frac{1}{q-1}\right) - \frac{1}{e} \left(n -\frac{1}{q-1} \right) = \vp(\pi). \] 
This implies that $\Tr_{F_{n+1}/F_n}(\OO_{F_{n+1}}) = \pi \OO_{F_n}$ 
by proposition~7 of Chapter~III of \cite{SCL}.

Since $\pi$ divides $q/\pi$, this shows that given $y \in \OO_{F_k}$, there exists a sequence $\{y_n\}_{n \geq 1}$ with $x_n \in \OO_{F_n}$ such that $y_k = y$, and $\Tr_{F_{n+1}/F_n}(y_{n+1}) = q/\pi \cdot y_n$ for $n \geq 1$. Take $\ell_1, \ell_2 \geq 0$ such that $\pi^{\ell_1} \OO_{\Cp}$ is in the domain of $\exp_{\LT}$ and such that $\pi^{\ell_2} \log_{\LT}(z) \in \OO_{F_k}$. Let $y = \pi^{\ell_2} \log_{\LT}(z)$. Let $\{y_n\}_{n \geq 1}$ be a sequence as above, let $x_n = \exp_{\LT}(\pi^{\ell_1} y_n)$ and $\ell=\ell_1+\ell_2$. The elements $x_k \ominus [\pi^\ell](z)$, as well as $\Tr^{\LT}_{F_{n+1}/F_n}(x_{n+1}) \ominus [q/\pi](x_n)$ for all $n$, have their $\log_{\LT}$ equal to zero and are in a domain in which $\log_{\LT}$ is injective. This proves the proposition.
\end{proof}

If $x \in S$ and $y_n = \log_{\LT}(x_n)$, then $y_n \in F_n$ and $\Tr_{F_{n+1} / F_n}(y_{n+1}) = q/\pi \cdot y_n$, so that by proposition \ref{interlog}, there exists $f(T) \in \brigplus{,F}$ such that $\psiq (f(T)) = \pi^{-1}  \cdot f(T)$ and $y_n = f(u_n)$ for all $n \geq 1$. If $f(T) \in \brigplus{,F}$ is such that $\psiq(f(T)) = \pi^{-1} \cdot f(T)$, then $\partial f \in (\brigplus{,F})^{\psiq=1}$ and $\partial f \cdot u$ can be seen as an element of $\drig{}(V)^{\psiq=1}$.

\begin{theo}
\label{expchif}
If $x \in S$, and if $f(T) \in \brigplus{,F}$ is such that $f(u_n) = \log_{\LT}(x_n)$ and $\psiq(f(T)) = \pi^{-1} \cdot f(T)$, then $h^1_{F_n,V} (\partial f(T) \cdot u) = (q/\pi)^{-n} \cdot \delta(x_n)$ for all $n \geq 1$.
\end{theo}

\begin{proof}
Let $y = f(T) \otimes t_\pi^{-1} u$, so that $y \in (\brigplus{,F} \otimes_F \dcris(V))^{\psiq=1}$. By theorem \ref{expbk} applied to $y$ with $h=1$, we have $h^1_{F_n,V}(\nabla (y)) = \exp_{F_n,V}(q^{-n} \partial_V(\phiq^{-n}(y)))$ if $n \geq 1$. Since $\phiq^{-n} \circ \partial = \pi^n \cdot \partial \circ \phiq^{-n}$, this implies that  
\[ h^1_{F_n,V} (\partial f(T) \cdot u) =\exp_{F_n,V}(q^{-n} \partial_V(\phiq^{-n}(y))) = (q/\pi)^{-n} \cdot \exp_{F_n,V}(\log_{\LT}(x_n) \cdot u). \]
By example 3.10.1 of \cite{BK90} and lemma \ref{expcompat}, we have $\delta(x_n) = \exp_{F_n,V}(\log_{\LT}(x_n) \cdot u)$. This proves the theorem.
\end{proof}

\begin{rema}
\label{colchecol}
If $F=\Qp$ and $\pi=q=p$ and $x = \{x_n\}_{n \geq 1}$, this theorem says that $\Exp^*_{\Qp} ( \delta(x) ) = \partial \log \Col_x(T)$, which is (iii) of proposition V.3.2 of \cite{CC99} (see theorem II.1.3 of ibid for the definition of the map $\Exp^*_{\Qp} : \Hiw^1(F,\Qp(1)) \to \drig{}(\Qp(1))^{\psiq=1}$).
\end{rema}

\begin{rema}
\label{nocolpow}
If $x \in S$, then by proposition \ref{interlog}, there is a power series $f(T)$ such that $f(u_n) = \log_{\LT}(x_n)$ for $n \geq 1$. Is there a power series $g(T) \in \OO_F\dcroc{T}$ such that $g(u_n) = x_n$, so that $f(T) = \log g(T)$?

If $F=\Qp$, such a power series is the classical Coleman power series \cite{RC79}. If $F \neq \Qp$ and $x \in S$ and $z$ is a $[q/\pi]$-torsion point, and $k \geq d-1$ so that $z \in F_k$, then the sequence $x' = \{x'_n\}_{n \geq 1}$ defined by $x_n'=x_n$ if $n \neq k$ and $x_k'=x_k \oplus z$ also belongs to $S$. This means that we cannot na\"{\i}vely interpolate $x$.
\end{rema}

\subsection{Perrin-Riou's big exponential map}
\label{bprfan}

In this last section, we explain how the explicit formulas of the previous sections can be used to give a Lubin-Tate analogue of Perrin-Riou's ``big exponential map'' \cite{BP94}. Take $h \geq 1$ such that $\Fil^{-h} \dcris(V) = \dcris(V)$. If $f \in \brigplus{,F} \otimes_F \dcris(V)$, let $\Delta(f)$ be the image of
$\oplus_{k=0}^h \partial^k(f)(0)$ in $\oplus_{k=0}^h \dcris(V)/(1-\pi^k \phiq)$.

\begin{lemm}
\label{colcol}
There is an exact sequence:
\begin{multline*}
 0 \to \oplus_{k=0}^h t_\pi^k \dcris(V)^{\phiq=\pi^{-k}} \to
\left( \brigplus{,F} \otimes_F \dcris(V) \right)^{\psiq=1} 
\xrightarrow{1-\phiq} \\ 
(\brigplus{,F})^{\psiq=0} \otimes_F \dcris(V)
\xrightarrow{\Delta} \oplus_{k=0}^h \frac{\dcris(V)}{1-\pi^k 
\phiq}\to 0.
\end{multline*}
\end{lemm}

\begin{proof}
Note that the map $\phiq$ acts diagonally on tensor products. It is easy to see that $\ker(1-\phiq) = \oplus_{k=0}^h t_\pi^k \dcris(V)^{\phiq=\pi^{-k}}$, that $\Delta$ is surjective,  and that $\mathrm{im}(1-\phiq) \subset \ker \Delta$, so we now prove that $\mathrm{im}(1-\phiq) = \ker \Delta$. 

If $f,g \in \brigplus{,F} \otimes_F \dcris(V)$ and $f=(1-\phiq)g$, then $\psiq(f)=0$ if and only if $\psiq(g)=g$. It is therefore enough to show that if $f \in \brigplus{,F} \otimes_F \dcris(V)$ is such that $\Delta(f)=0$, then $f=(1-\phiq)g$ for some $g \in \brigplus{,F} \otimes_F \dcris(V)$. 

The map $1-\phiq : T^{h+1} \brigplus{,F} \otimes_F \dcris(V)
\to T^{h+1} \brigplus{,F} \otimes_F \dcris(V)$ is bijective because the slopes of $\phiq$ on $T^{h+1} \brigplus{,F} \otimes_F D$ are $>0$. This implies that $1-\phiq$ induces a sequence

\begin{multline*}
0 \to \oplus_{k=0}^h t_\pi^k \dcris(V)^{\phiq=\pi^{-k}} \to 
\frac{\brigplus{,F} \otimes_F \dcris(V)}{T^{h+1} \brigplus{,F} \otimes_F \dcris(V)}
\xrightarrow{\overline{1-\phiq}} \\
\frac{\brigplus{,F} \otimes_F \dcris(V)}{T^{h+1} \brigplus{,F} \otimes_F \dcris(V)} 
\xrightarrow{\Delta} \oplus_{k=0}^h \frac{\dcris(V)}{1-\pi^k 
\phiq}.
\end{multline*} 
We have $\ker(\overline{1-\phiq}) = \oplus_{k=0}^h t_\pi^k \dcris(V)^{\phiq=\pi^{-k}}$ and by comparing dimensions, we see that $\coker(\overline{1-\phiq}) = \oplus_{k=0}^h \dcris(V) / (1-\pi^k 
\phiq)$. This and the bijectivity of $1-\phiq$ on $T^{h+1} \brigplus{,F} \otimes_F \dcris(V)$ imply the claim.
\end{proof}

If $f \in ((\brigplus{,F})^{\psiq=0} \otimes_F \dcris(V))^{\Delta=0}$, 
then by lemma \ref{colcol}
there exists $y \in ( \brigplus{,F} \otimes_F \dcris(V) )^{\psiq=1}$ 
such that $f = (1-\phiq)y$. Since 
$\nabla_{h-1} \circ \cdots \circ \nabla_0$ kills 
$\oplus_{k=0}^{h-1} t_\pi^k \dcris(V)^{\phiq=\pi^{-k}}$ we see that
$\nabla_{h-1} \circ \cdots \circ \nabla_0 (y)$ does not depend upon
the choice of such a $y$ (unless $\dcris(V)^{\phiq=\pi^{-h}} \neq 0$). 

\begin{defi}\label{expbpr}
Let $h \geq 1$ be such that $\Fil^{-h} \dcris(V) =
\dcris(V)$ and such that $\dcris(V)^{\phiq=\pi^{-h}}=0$. 
We deduce from the above construction a well-defined map:
\[ \Omega_{V,h}:
((\brigplus{,F})^{\psiq=0} \otimes_F \dcris(V))^{\Delta=0} \to
\drig{}(V)^{\psiq=1}, \]
given by
$\Omega_{V,h}(f) = \nabla_{h-1} \circ \cdots \circ \nabla_0 (y)$ where the element 
$y \in ( \brigplus{,F} \otimes_F \dcris(V) )^{\psiq=1}$ 
is such that $f = (1-\phiq)y$ and is provided by lemma \ref{colcol}. 

If $\dcris(V)^{\phiq=\pi^{-h}} \neq 0$, we get a map
\[ \Omega_{V,h}:
((\brigplus{,F})^{\psiq=0} \otimes_F \dcris(V))^{\Delta=0} \to
\drig{}(V)^{\psiq=1}/ V^{G_F=\chilt^h}. \]
\end{defi}

Let $u$ be a basis of $F(\chilt)$ as above, and let $e_j = u^{\otimes j}$ if $j \in \ZZ$.

\begin{theo}
\label{imgbpr}
Take $y \in (\brigplus{,F} \otimes_F \dcris(V))^{\psiq=1}$ and let $h \geq 1$ be such that
$\Fil^{-h}\dcris(V)=\dcris(V)$. Let $f=(1-\phiq)y$ so that $f \in ((\brigplus{,F})^{\psiq=0} \otimes_F \dcris(V))^{\Delta=0}$.

If $j \in \ZZ$ and $h+j \geq 1$, then
\begin{multline*}
h^1_{F_n,V(\chilt^j)}(\Omega_{V,h}(f) \otimes
e_j) = 
(-1)^{h+j-1} (h+j-1)! 
 \times \\ \begin{cases}
\exp_{F_n,V(\chilt^j)}(q^{-n} \partial_{V(\chilt^j)}(\phiq^{-n} 
(\partial^{-j}y \otimes t_\pi^{-j}e_j)))
& \text{if $n \geq 1$} \\
\exp_{F,V(\chilt^j)}((1-q^{-1} \phiq^{-1})\partial_{V(\chilt^j)} 
(\partial^{-j}y \otimes t_\pi^{-j}e_j))
& \text{if $n=0$.}
\end{cases} 
\end{multline*}
If $j \in \ZZ$ and $h+j \leq 0$, then
\begin{multline*}
\exp^*_{F_n,V^*(1-j)}(h^1_{F_n,V(\chilt^j)}(\Omega_{V,h}(f) \otimes
e_j))=
\\
\frac{1}{(-h-j)!}
\begin{cases} 
q^{-n} \partial_{V(\chilt^j)}(\phiq^{-n} 
(\partial^{-j}y \otimes t_\pi^{-j}e_j))
& \text{if $n \geq 1$} \\
(1-q^{-1} \phiq^{-1})\partial_{V(\chilt^j)} 
(\partial^{-j}y \otimes t_\pi^{-j}e_j)
& \text{if $n = 0$.}
\end{cases}
\end{multline*}
\end{theo}

\begin{proof}
If $h+j \geq 1$, the following diagram is commutative:
\[ \begin{CD}
\drig{}(V)^{\psiq=1} @>{\otimes e_j}>> \drig{}(V(\chilt^j))^{\psiq=1}  \\
@A{\nabla_{h-1} \circ \cdots \circ \nabla_0}AA 
@A{\nabla_{h+j-1} \circ \cdots \circ \nabla_0}AA \\
\left( \brigplus{,F} \otimes_F \dcris(V) \right)^{\psiq=1}
@>{\partial^{-j} \otimes t^{-j} e_j}>>
\left( \brigplus{,F} \otimes_F \dcris(V(\chilt^j)) \right)^{\psiq=1},
\end{CD} \]
and the theorem is a straightforward consequence of theorem 
\ref{expbk} applied to $\partial^{-j}y \otimes t^{-j}e_j$, $h+j$
and $V(\chilt^j)$ (which are the $j$-th twists of $y$, $h$ and $V$).

If  $h+j \leq 0$, and $\Gamma_{F_n}$ is torsion free, 
then theorem \ref{interdual} shows that
\begin{multline*}
\exp^*_{F_n,V^*(1-j)}
(h^1_{F_n,V(\chilt^j)}(\nabla_{h-1} \circ \cdots 
\circ \nabla_0 (y) \otimes e_j))  \\ = 
q^{-n} \partial_{V(\chilt^j)}(\phiq^{-n}(\nabla_{h-1} \circ \cdots 
\circ \nabla_0 (y) \otimes e_j)) 
\end{multline*}
in $\dcris(V(\chilt^j))$, 
and a short computation involving Taylor series shows that 
\[
\partial_{V(\chilt^j)}(\phiq^{-n}(\nabla_{h-1} \circ \cdots 
\circ \nabla_0 (y) \otimes e_j)) = 
(-h-j)!^{-1} \partial_{V(\chilt^j)}(\phiq^{-n} 
(\partial^{-j}y \otimes t_\pi^{-j}e_j)). 
\]
To get the other $n$, we  corestrict.
\end{proof}

\begin{coro}
\label{twpr}
We have $\Omega_{V,h}(x) \otimes e_j = \Omega_{V(\chilt^j),h+j}(\partial^{-j}x
\otimes t_\pi^{-j}e_j)$ and $\nabla_h \circ \Omega_{V,h}(x) = \Omega_{V,h+1}(x)$.
\end{coro}

\begin{rema}
The notation $\partial^{-j}$ is somewhat abusive if $j \geq 1$ as
$\partial$ is not injective on $\brigplus{,F}$
(it is surjective as can be seen by
``integrating'' directly a power series) but the reader can check 
that this leads to no ambiguity in the formulas of theorem
\ref{imgbpr} above. 
\end{rema}

If $F=\Qp$ and $\pi=p$, definition \ref{expbpr} and theorem \ref{imgbpr} are given in \S II.5 of \cite{LB3}. They imply that $\Omega_{V,h}$ co\"{\i}ncides with Perrin-Riou's exponential map (see theorem 3.2.3 of \cite{BP94}) after making suitable identifications (theorem II.13 of \cite{LB3}). 

Our definition therefore generalizes Perrin-Riou's exponential map to the $F$-analytic setting. We hope to use the results of \cite{F05} and \cite{F08} to relate our constructions to suitable Iwasawa algebras as in the cyclotomic case.

\providecommand{\bysame}{\leavevmode ---\ }
\providecommand{\og}{``}
\providecommand{\fg}{''}
\providecommand{\smfandname}{\&}
\providecommand{\smfedsname}{\'eds.}
\providecommand{\smfedname}{\'ed.}
\providecommand{\smfmastersthesisname}{M\'emoire}
\providecommand{\smfphdthesisname}{Th\`ese}


\begin{thebibliography}{FW79b}

\bibitem[Ber02]{LB2}
{\scshape L.~Berger} -- {\og Repr{\'e}sentations {$p$}-adiques et \'equations
  diff\'erentielles\fg}, \emph{Invent. Math.} \textbf{148} (2002), no.~2,
  p.~219--284.

\bibitem[Ber03]{LB3}
\bysame , {\og Bloch and {K}ato's exponential map: three explicit formulas\fg},
  \emph{Doc. Math.} (2003), no.~Extra Vol., p.~99--129 (electronic), Kazuya
  Kato's fiftieth birthday.

\bibitem[Ber08]{LB8}
\bysame , {\og Construction de {$(\varphi,\Gamma)$}-modules:
  repr{\'e}sentations {$p$}-adiques et {$B$}-paires\fg}, \emph{Algebra Number
  Theory} \textbf{2} (2008), no.~1, p.~91--120.

\bibitem[Ber16]{PGMLAV}
\bysame , {\og Multivariable {$(\varphi,\Gamma)$}-modules and locally analytic
  vectors\fg}, \emph{Duke Math. J.} \textbf{165} (2016), no.~18, p.~3567--3595.

\bibitem[BK90]{BK90}
{\scshape S.~Bloch {\normalfont \smfandname} K.~Kato} -- {\og {$L$}-functions
  and {T}amagawa numbers of motives\fg}, in \emph{The {G}rothendieck
  {F}estschrift, {V}ol.\ {I}}, Progr. Math., vol.~86, Birkh\"auser Boston,
  Boston, MA, 1990, p.~333--400.

\bibitem[CC98]{CC98}
{\scshape F.~Cherbonnier {\normalfont \smfandname} P.~Colmez} -- {\og
  Repr\'esentations {$p$}-adiques surconvergentes\fg}, \emph{Invent. Math.}
  \textbf{133} (1998), no.~3, p.~581--611.

\bibitem[CC99]{CC99}
{\scshape F.~Cherbonnier {\normalfont \smfandname} P.~Colmez} -- {\og Th\'eorie
  d'{I}wasawa des repr\'esentations {$p$}-adiques d'un corps local\fg},
  \emph{J. Amer. Math. Soc.} \textbf{12} (1999), no.~1, p.~241--268.

\bibitem[Col79]{RC79}
{\scshape R.~F. Coleman} -- {\og Division values in local fields\fg},
  \emph{Invent. Math.} \textbf{53} (1979), no.~2, p.~91--116.

\bibitem[Col02]{C02}
{\scshape P.~Colmez} -- {\og Espaces de {B}anach de dimension finie\fg},
  \emph{J. Inst. Math. Jussieu} \textbf{1} (2002), no.~3, p.~331--439.

\bibitem[Col16]{LT}
\bysame , {\og Repr\'esentations localement analytiques de
  {${\mathrm{GL}}_2(\mathbf{Q}_p)$} et {$(\varphi,\Gamma)$}-modules\fg},
  \emph{Represent. Theory} \textbf{20} (2016), p.~187--248.

\bibitem[Fon90]{F90}
{\scshape J.-M. Fontaine} -- {\og Repr\'esentations {$p$}-adiques des corps
  locaux. {I}\fg}, in \emph{The {G}rothendieck {F}estschrift, {V}ol.\ {II}},
  Progr. Math., vol.~87, Birkh\"auser Boston, Boston, MA, 1990, p.~249--309.

\bibitem[Fon94]{FPP}
\bysame , {\og Le corps des p\'eriodes {$p$}-adiques\fg}, \emph{Ast\'erisque}
  (1994), no.~223, p.~59--111, With an appendix by Pierre Colmez, P{\'e}riodes
  $p$-adiques (Bures-sur-Yvette, 1988).

\bibitem[Fou05]{F05}
{\scshape L.~Fourquaux} -- {\og Logarithme de {P}errin-{R}iou pour des
  extensions associ\'ees \`a un groupe de {L}ubin-{T}ate\fg},
  \smfphdthesisname, Universit\'e {P}aris 6, 2005.

\bibitem[Fou08]{F08}
\bysame , {\og Logarithme de {P}errin-{R}iou pour des extensions associ\'ees
  \`a un groupe de {L}ubin-{T}ate\fg}, preprint, 2008.

\bibitem[FW79a]{FW2}
{\scshape J.-M. Fontaine {\normalfont \smfandname} J.-P. Wintenberger} -- {\og
  Extensions alg\'ebriques et corps des normes des extensions {APF} des corps
  locaux\fg}, \emph{C. R. Acad. Sci. Paris S\'er. A-B} \textbf{288} (1979),
  no.~8, p.~A441--A444.

\bibitem[FW79b]{FW1}
\bysame , {\og Le ``corps des normes'' de certaines extensions alg\'ebriques de
  corps locaux\fg}, \emph{C. R. Acad. Sci. Paris S\'er. A-B} \textbf{288}
  (1979), no.~6, p.~A367--A370.

\bibitem[FX13]{FX13}
{\scshape L.~Fourquaux {\normalfont \smfandname} B.~Xie} -- {\og Triangulable
  {$\mathcal{O}_F$}-analytic {$(\varphi_q,\Gamma)$}-modules of rank $2$\fg},
  \emph{Algebra Number Theory} \textbf{7} (2013), no.~10, p.~2545--2592.

\bibitem[Iwa86]{IK}
{\scshape K.~Iwasawa} -- \emph{Local class field theory}, Oxford Science
  Publications, The Clarendon Press, Oxford University Press, New York, 1986,
  Oxford Mathematical Monographs.

\bibitem[Kat93]{K93}
{\scshape K.~Kato} -- {\og Lectures on the approach to {I}wasawa theory for
  {H}asse-{W}eil {$L$}-functions via {$B\sb {\rm dR}$}. {I}\fg}, in
  \emph{Arithmetic algebraic geometry ({T}rento, 1991)}, Lecture Notes in
  Math., vol. 1553, Springer, Berlin, 1993, p.~50--163.

\bibitem[KR09]{KR09}
{\scshape M.~Kisin {\normalfont \smfandname} W.~Ren} -- {\og Galois
  representations and {L}ubin-{T}ate groups\fg}, \emph{Doc. Math.} \textbf{14}
  (2009), p.~441--461.

\bibitem[Laz62]{L62}
{\scshape M.~Lazard} -- {\og Les z\'eros des fonctions analytiques d'une
  variable sur un corps valu\'e complet\fg}, \emph{Inst. Hautes \'Etudes Sci.
  Publ. Math.} (1962), no.~14, p.~47--75.

\bibitem[LT65]{LT65}
{\scshape J.~Lubin {\normalfont \smfandname} J.~Tate} -- {\og Formal complex
  multiplication in local fields\fg}, \emph{Ann. of Math. (2)} \textbf{81}
  (1965), p.~380--387.

\bibitem[Mat95]{SM95}
{\scshape S.~Matsuda} -- {\og Local indices of {$p$}-adic differential
  operators corresponding to {A}rtin-{S}chreier-{W}itt coverings\fg},
  \emph{Duke Math. J.} \textbf{77} (1995), no.~3, p.~607--625.

\bibitem[PR94]{BP94}
{\scshape B.~Perrin-Riou} -- {\og Th\'eorie d'{I}wasawa des repr\'esentations
  {$p$}-adiques sur un corps local\fg}, \emph{Invent. Math.} \textbf{115}
  (1994), no.~1, p.~81--161, With an appendix by Jean-Marc Fontaine.

\bibitem[Ser68]{SCL}
{\scshape J.-P. Serre} -- \emph{Corps locaux}, Hermann, Paris, 1968,
  Deuxi{\`e}me {\'e}dition, Publications de l'Universit{\'e} de Nancago, No.
  VIII.

\bibitem[Ser94]{SCG}
\bysame , \emph{Cohomologie galoisienne}, fifth \smfedname, Lecture Notes in
  Mathematics, vol.~5, Springer-Verlag, Berlin, 1994.

\bibitem[SV17]{SV}
{\scshape P.~Schneider {\normalfont \smfandname} O.~Venjakob} -- {\og
  Coates-{W}iles homomorphisms and {I}wasawa cohomology for {L}ubin-{T}ate
  extensions\fg}, in \emph{Elliptic Curves, Modular Forms and Iwasawa Theory.
  In Honour of John H. Coates' 70th Birthday}, Springer Verlag, 2017.

\bibitem[Tam15]{GT}
{\scshape G.~Tamme} -- {\og On an analytic version of {L}azard's
  isomorphism\fg}, \emph{Algebra Number Theory} \textbf{9} (2015), no.~4,
  p.~937--956.

\bibitem[Tsu04]{TT}
{\scshape T.~Tsuji} -- {\og Explicit reciprocity law and formal moduli for
  {L}ubin-{T}ate formal groups\fg}, \emph{J. Reine Angew. Math.} \textbf{569}
  (2004), p.~103--173.

\bibitem[Win83]{WCN}
{\scshape J.-P. Wintenberger} -- {\og Le corps des normes de certaines
  extensions infinies de corps locaux; applications\fg}, \emph{Ann. Sci.
  \'Ecole Norm. Sup. (4)} \textbf{16} (1983), no.~1, p.~59--89.

\end{thebibliography}
\end{document}